\documentclass[bst/sn-mathphys-num]{sn-jnl}

\usepackage{graphicx}%
\usepackage{multirow}%
\usepackage{amsmath,amssymb,amsfonts}%
\usepackage{amsthm}%
\usepackage{mathrsfs}%
\usepackage[title]{appendix}%
\usepackage{xcolor}%
\usepackage{textcomp}%
\usepackage{manyfoot}%
\usepackage{booktabs}%
\usepackage{algorithm}%
\usepackage{algorithmicx}%
\usepackage{algpseudocode}%
\usepackage{listings}%

\usepackage{float}
\usepackage{comment}

\usepackage{tabularx}

\usepackage{mathtools}
\usepackage{tcolorbox}
\usepackage{bbm}
\usepackage{todonotes}
\usepackage{tensor}
\usepackage[font=small,labelfont=bf]{caption}
\usepackage{subcaption}
\usepackage{graphicx}

\usepackage[noadjust]{cite}

\newcommand{\RR}{\mathbb{R}}

\usepackage{accents}

\mathchardef\mhyphen="2D 
\DeclareMathOperator*{\argmin}{\arg\!\min}

\DeclarePairedDelimiter{\norm}{\lVert}{\rVert}

\newtheorem{lemma}{Lemma}[section]

\usepackage[nameinlink]{cleveref}
\crefname{equation}{}{}
\crefname{theorem}{Theorem}{Theorems}
\crefname{corollary}{Corollary}{Corollaries}
\crefname{example}{Example}{Examples}
\crefname{assumption}{Assumption}{Assumptions}
\crefname{lemma}{Lemma}{Lemmas}
\crefname{proposition}{Proposition}{Propositions}
\crefname{figure}{Figure}{Figures}
\crefname{table}{Table}{Tables}
\crefname{section}{Section}{Sections}
\crefname{appendix}{Appendix}{Appendices}
\Crefname{equation}{}{}
\Crefname{theorem}{Theorem}{Theorems}
\Crefname{corollary}{Corollary}{Corollaries}
\Crefname{example}{Example}{Examples}
\Crefname{lemma}{Lemma}{Lemma}
\Crefname{proposition}{Proposition}{Propositions}
\Crefname{figure}{Figure}{Figures}
\Crefname{table}{Table}{Tables}
\Crefname{section}{Section}{Sections}
\Crefname{appendix}{Appendix}{Appendices}

\usepackage{titlesec} 
\titleformat{\subsubsection}[runin]{\normalfont\bfseries}{\thesubsubsection.}{10pt}{}

\usepackage{adjustbox} 

\let\originalleft\left
\let\originalright\right
\renewcommand{\left}{\mathopen{}\mathclose\bgroup\originalleft}
\renewcommand{\right}{\aftergroup\egroup\originalright}

\ifpdf
\hypersetup{
  pdftitle={Error Bounds, PL Condition, and Quadratic Growth for Weakly Convex Functions, and Linear Convergences of Proximal Point Methods},
  pdfauthor={Feng-Yi Liao, Lijun Ding, and Yang Zheng}
}
\fi

\usepackage{caption}

\newcommand{\tr}{\mathsf{ T}}

\newcommand{\Dist}{\mathrm{dist}}
\newcommand{\xk}{x_k}
\newcommand{\xknext}{x_{k+1}}

\newcommand{\innerproduct}[2]{\left \langle #1, #2 \right\rangle }

\newcommand{\mup}{\mu_{\mathrm{p}}} 
\newcommand{\muq}{\mu_{\mathrm{q}}} 
\newcommand{\mue}{\mu_{\mathrm{e}}} 
\newcommand{\mus}{\mu_{\mathrm{s}}} 
\newcommand{\mur}{\mu_{\mathrm{r}}} 

\newcommand{\fpartial}{\hat{\partial}}

\usepackage{multirow}

\usepackage[utf8]{inputenc}
\usepackage[english]{babel}
\usepackage{xcolor}

\newcommand{\prox}{\text{prox}}

\theoremstyle{thmstyleone}%
\newtheorem{theorem}{Theorem}[section]

\theoremstyle{thmstyletwo}%
\newtheorem{example}{Example}%
\newtheorem{remark}{Remark}%

\theoremstyle{thmstylethree}%
\newtheorem{definition}{Definition}%

\begin{document}

\title[Error Bounds, PL Condition, and Quadratic Growth for Weakly Convex Functions, and Linear Convergences of Proximal Point Methods]{Error Bounds, PL Condition, and Quadratic Growth for Weakly Convex Functions, and Linear Convergences of Proximal Point Methods\footnote{This paper  
is an extended version of the conference paper \cite{pmlr-v242-liao24a}.}}


\author[1]{\fnm{Feng-Yi} \sur{Liao}}\email{fliao@ucsd.edu}

\author[2]{\fnm{Lijun} \sur{Ding}}\email{l2ding@ucsd.edu}

\author[1]{\fnm{Yang} \sur{Zheng}}\email{zhengy@ucsd.edu}

\affil[1]{\orgdiv{Department of Electrical and Computer Engineering}, \orgname{University of California San Diego}, 
\orgaddress{
\state{California},
\country{United state}}}

\affil[2]{\orgdiv{Department of Mathematics}, \orgname{University of California San Diego}, \orgaddress{
\state{California}, 
\country{United state}}}



\abstract{
     
  Many practical optimization problems lack strong convexity. Fortunately, recent studies have revealed that first-order algorithms also enjoy linear convergences under various weaker regularity conditions. While the relationship among different conditions for convex and smooth functions is well-understood, it is not the case for the nonsmooth setting. In this paper, we go beyond convexity and smoothness, and clarify the connections among common regularity conditions in the class of weakly convex functions, including \textit{strong convexity}, \textit{restricted secant inequality}, \textit{subdifferential error bound}, \textit{Polyak-Łojasiewicz inequality}, and \textit{quadratic growth}.
    In addition, using these regularity conditions, we present a simple and modular proof for the linear convergence of {proximal point method} (PPM)  for convex and weakly convex optimization problems. The linear convergence also holds when the subproblems of PPM are solved inexactly with a proper control of inexactness.  
  }

\keywords{Error bound, Polyak-Łojasiewicz inequality, Quadratic growth, Proximal point method, Linear convergence, Weakly convex functions}



\maketitle

\section{Introduction}
With the success of machine learning, (sub)gradient-type algorithms, such as (sub)gradient descent, coordinate descent, Frank-Wolf algorithm, and other variants, have gained much attention over the past decades \cite{beck2017first,bottou2018optimization}. 
For smooth and/or convex cases, (sub)gradient methods are most well-understood. It is well-known that the basic gradient descent algorithm achieves linear convergence for minimizing smooth and strongly convex functions \cite{nesterov2018lectures}.    
However, strong convexity is often not satisfied in practice, and many fundamental models in machine learning lack this good property~\cite{agarwal2010fast}. 

This motivates the study of different regularity conditions that are weaker than strong convexity, such as Polyak-Łojasiewicz inequality \cite{polyak1963gradient}, restricted secant inequality \cite{zhang2013gradient}, subdifferential error bound \cite{drusvyatskiy2018error}, and quadratic growth \cite{necoara2019linear}. These regularity conditions go beyond strong convexity, and they mainly involve the relationship concerning the distance to the solution set, the cost value gap, and the length of the (sub)gradient. They often play a crucial step in many algorithm analyses. It has been shown that under those regularity conditions, a range of first-order optimization algorithms still enjoy linear convergence in many cases; see e.g.,  \cite{polyak1963gradient,karimi2016linear,cui2019r,drusvyatskiy2018error,diaz2023optimal,garber2023linear,liao2023overview,ding2023revisiting}. Moreover, these convergence analyses even become more elegant under suitable regularity conditions; examples include the gradient descent method for smooth optimization \cite[Theorem 2.1.15]{nesterov2018lectures}; the proximal gradient method for nonsmooth convex optimization \cite[Theorem 3.2]{drusvyatskiy2018error}; and the proximal point method (PPM) to solve 
 $ 0 \in T(x)$ where $T$ is a maximal monotone set value mapping \cite{leventhal2009metric}.

Recent studies have revealed the relationship between some regularity conditions under various assumptions on the smoothness and convexity of the objective function \cite{karimi2016linear,drusvyatskiy2018error,artacho2008characterization,bolte2017error,ye2021variational,zhu2023unified,drusvyatskiy2021nonsmooth,guille2021study,rebjock2023fast}. We postpone the comparison between different studies after we formally define regularity conditions (see \Cref{remark:connections}). Different conditions are useful to prove the convergence of various algorithms. For example,  the subdifferential error bound yields a clean proof for the proximal point method (PPM) \cite{rockafellar1976monotone,leventhal2009metric}; the step-length type error bound becomes useful for the proximal gradient method \cite{drusvyatskiy2018error}; Polyak-Łojasiewicz inequality leads to a clean proof of the gradient descent \cite{polyak1963gradient,karimi2016linear}; quadratic growth plays an important role in the analysis of proximal bundle methods \cite{diaz2023optimal} and spectral bundle methods \cite{ding2023revisiting,liao2023overview}. In some cases, it may be easier to establish one regularity condition than another one and then use the equivalence between the two to derive convergence results,  e.g. \cite[Proposition 3.3]{cui2016asymptotic}. Therefore, understanding the relationship or equivalence between common regularity conditions becomes crucial in algorithm analyses. 

Despite the results mentioned above on the equivalence \cite{karimi2016linear,drusvyatskiy2018error,artacho2008characterization,bolte2017error,ye2021variational,zhu2023unified,drusvyatskiy2021nonsmooth,guille2021study,rebjock2023fast}, their scope remains limited to one of the following scenarios: 1) smooth functions, or 2) nonsmooth but convex functions. While smooth and convex (but not strongly convex) problems cover a variety of applications, practical applications routinely deal with problems that lack both qualities (e.g., training nonsmooth and nonconvex deep networks). 
Recent studies have further identified another amenable problem
class: \textit{weakly convex functions} \cite{davis2019stochastic}. This class of functions includes all convex functions, $L$-smooth functions, certain compositions of convex functions with smooth functions, and many cost functions in modern machine learning \cite{drusvyatskiy2020subgradient,atenas2023unified}. Some algorithms for convex functions can also be extended to solve weakly convex functions, e.g., the classical bundle method \cite{liang2023proximal}, the stochastic projected subgradient method \cite{davis2019proximally}, and the damped proximal augmented Lagrangian method \cite{dahal2023damped}. A recent study also proposes {a unified analysis framework} to analyze the convergence of algorithms designed for weakly convex functions \cite{atenas2023unified}. In this paper, we aim to generalize the previous results on the equivalence of different regularity conditions \cite{karimi2016linear,drusvyatskiy2018error,artacho2008characterization,bolte2017error,ye2021variational,zhu2023unified,drusvyatskiy2021nonsmooth} from the class of smooth functions (and the class of convex functions) to the class of nonsmooth weakly convex functions. These findings will benefit algorithm analyses from various perspectives.

To demonstrate the equivalence between different regularity conditions, we revisit the classical (inexact) PPM, a conceptually simple algorithm that serves as a guideline for many algorithm developments \cite{drusvyatskiy2017proximal}. In particular, we provide a simple new proof for the linear convergence in both the convex and weakly convex settings using Polyak-Łojasiewicz inequality, subdifferential error bound, and quadratic growth. Our proof is partially inspired by a recent result on spectral bundle methods \cite{ding2023revisiting,liao2023overview} and the analysis in \cite{leventhal2009metric}. As a by-product of our investigation, we clarify some subtleties in inexact PPM in terms of the inexactness and feasibility, which are less emphasized in the literature. In this paper, our contributions are as follows.   
\begin{itemize}
    \item We first clarify the relationship among common regularity conditions, including \textit{strongly convex}, \textit{restricted secant inequality}, \textit{subdifferential error bound}, \textit{Polyak-Łojasiewicz inequality}, and \textit{quadratic growth}, in the class of weakly convex functions (\Cref{prop:QG-EB-PL}). Most of the proofs follow from some elementary analysis. The connection between  Polyak-Łojasiewicz inequality and subdifferential error bound
    utilizes the notion of slope in \cite[Section 2]{drusvyatskiy2021nonsmooth} and requires some sophisticated arguments.We also derive the coefficient {conversion} among different conditions (see \cref{table:equivalent-condtions}). 
    \item Building upon the equivalence for different regularity conditions, we present a simple and modular proof for linear convergence of the PPM for convex optimization \cite{rockafellar1976monotone} (\Cref{thm-linear}). The linear convergence also extends to the class of weakly convex functions with a proper initialization (\cref{thm:linear-weakly}). Our linear convergence result allows non-unique minimizers which is more general than a classical assumption of local Lipschitz continuity in \cite{rockafellar1976monotone} and also relaxes an upper Lipschitz condition in \cite{luque1984asymptotic}. 
    \item Finally, we clarify some subtleties in the inexact PPM regarding inexactness stopping criteria and feasibility (\cref{subsection:ippm-stopping-criteria}). We further establish the linear convergence of inexact PPM (\Cref{thm-linear-inexact}) under standard assumptions on the inexactness of the subproblem. Our proof is clean and modular, which combines a key inequality in \cite{luque1984asymptotic} with \Cref{thm-linear}. 
\end{itemize}
A preliminary summary of this work appeared in \cite{pmlr-v242-liao24a}. In this paper, we 1) complete the detailed proofs for all theorems and derive coefficient {conversion} among different conditions in \cref{table:equivalent-condtions}, 2) extend the linear convergence result of the PPM from optimization on convex functions to weakly convex functions (\cref{thm:linear-weakly}), and 3) provide extra technical discussions on inexact PPM (\cref{subsection:ippm-stopping-criteria}). The rest of this paper is structured as follows. \cref{section:Motivation-Preliminaries} summarizes the linear convergence of gradient descent as motivation and introduces some preliminaries in nonsmooth optimization techniques. \cref{section:equivalent-conditions} presents the relationship among different regularity conditions. \cref{section:sublinear-linear-PPM,section:inexact-PPM} focus on the (inexact) PPM and establish the sublinear and linear convergences. \cref{section:applications} presents three numerical experiments. \cref{section:conclusion} concludes this paper. Some auxiliary proofs and discussions are provided in \Cref{section:proof-inexact-PPM,appendix:background,section:proof-inexact-PPM,section:sublinear-PPM-proof,appendix:Section-2}. 

\vspace{1mm}

\noindent \textbf{Notation}.
We use $\RR^n$ to denote $n$-dimensional Euclidean space and $\overline{\RR}$ to denote the extended real line, i.e., $\overline{\RR}:= \RR \cup \{\pm\infty\}$. The notations $\innerproduct{\cdot}{\cdot}$ and $\|\cdot\|$ stand for standard inner product and $\ell_2$ norm in $\RR^n.$ For a closed set $S \subseteq \RR^n$, the distance of a point $x\in \RR^n$ to $S$ is defined as $\Dist(x,S) := \min_{y \in S} \|x - y\|$ and the projection of $x$ onto $S$ is denoted as $\Pi_{S}(x) = \argmin_{y\in S}\|x - y\|$. Given a function $f:\mathbb{R}^n \to \overline{\RR}$, the symbol $[f \leq \nu ]:= \{x \in \RR^n \mid f(x) \leq \nu \}$ denotes the $\nu$-sublevel set of $f$.
  
\section{Motivation and preliminaries}
    \label{section:Motivation-Preliminaries}
In this section, we first motivate the study of equivalent regularity conditions by revisiting the standard gradient descent algorithm, and then review some techniques in nonsmooth optimization. 
\subsection{Motivation: Linear convergence of gradient descent algorithms}
\label{subsection:motivation}
To motivate our discussion, consider a smooth convex optimization problem $\min_{x} \, f(x)$, 
where $f: \RR^n \to {\RR}$ is a convex and $L$-smooth function, i.e., its gradient is $L$-Lipschitz  satisfying
$
    \|\nabla f(x) - \nabla f(y)\| \leq L \|x  - y \|,\; \forall x,$ $ y \in \RR^n.
$
Let $S := \argmin_{x}\, f(x)$ be the set of optimal solutions. Assume $S \neq \emptyset$ and denote $f^\star = \min_{x} f(x)$. The standard \textit{gradient descent} (GD) follows the update 
\begin{equation} \label{eq:GD_iterates}
    \xknext = \xk - t_k \nabla f(\xk),
\end{equation}
where $t_k >0$ is the step size. A textbook result says that when choosing a constant step size $t_k = \frac{1}{L}$,  the GD iterates converge to $f^\star$ with a \textit{sublinear} rate (precisely, $f(x_k) -f^\star \leq L\Dist^2 (x_0, S)/(2k)$); see e.g., \cite[Corollary 2.1.2]{nesterov2018lectures}. If the function $f$ is strongly convex, GD achieves a global linear convergence \cite[Theorem 2.1.15]{nesterov2018lectures}.

However, the assumption of strong convexity is often not satisfied in practice. It is known that some alternative weaker assumptions are sufficient for linear convergences. We here introduce two notions: \textit{restricted secant inequality} (RSI) \cite{zhang2013gradient}, and \textit{Polyak-Łojasiewicz} (PL) inequality \cite{polyak1963gradient}. A differentiable function $f$ satisfies RSI if there exists $\mur > 0$ such that  
\begin{equation}
    \begin{aligned}
        \label{eq:RSI-smooth}
       \innerproduct{\nabla f(x)}{x - {x}^\star } \geq \mur \|x - {x}^\star \|^2 = \mur \cdot \Dist^2(x,S), \forall x \in \RR^n, {x}^\star \in \Pi_{S}(x),
    \end{aligned}
\end{equation}
and it satisfies the PL inequality if there exists $\mup > 0$ such that
\begin{equation}
    \begin{aligned}
        \label{eq:PL-smooth}
     \|\nabla f(x)\|^2 \geq 2\mup ( f(x) - f^\star), \; \forall x \in \RR^n.
    \end{aligned}
\end{equation}

Note that both RSI \cref{eq:RSI-smooth} and PL \cref{eq:PL-smooth} imply that any stationary point of $f$ is a global minimum. However, they do not imply the uniqueness of stationary points or the convexity of the function. One can think that RSI \cref{eq:RSI-smooth} (resp. PL \cref{eq:PL-smooth}) requires that the gradient $\nabla f(x)$ grows faster than that of quadratic functions when moving away from the solution set $S$ (resp. the optimal value $f^\star$). Linear convergence
of GD under the PL inequality was first proved in \cite{polyak1963gradient}, and linear convergence under RSI was discussed in \cite[Proposition 1]{guille2022gradient} and \cite[Proposition 1]{zhang2020new}. We summarize a simple version below.

\begin{theorem}[Linear convergence of GD]\label{proposition:GD-linear-convergence}
Consider the optimization problem $\min_{x \in \RR^n} f(x),$ where $f$ is an $L$-smooth function. Suppose its solution set $S$ is nonempty. If RSI \cref{eq:RSI-smooth} holds with $\mur > 0$ and PL inequality \cref{eq:PL-smooth} holds with $\mup > 0$, then the GD algorithm \cref{eq:GD_iterates} with a constant stepsize $t_k = \frac{\mur}{L^2}$ has global linear convergence for iterates and function values, i.e.,
\begin{subequations}
    \begin{align}
         \Dist(\xknext,S) &\leq \omega_1 \cdot  \Dist(\xk,S), \label{eq:linear-convergence-GD-iterates} \\ 
         f(\xknext) -f^\star & \leq \omega_2 \cdot (f(\xk) -f^\star), \label{eq:linear-convergence-GD-costs}
        \end{align}
\end{subequations}
with $\omega_1 = \sqrt{1-\mur^2/L^2} \in (0,1)$ and $\omega_2 = (L^3 - 2\mur L \mup + \mur^2 \mup)/L^3 \in (0,1)$.
\end{theorem}
 
Thanks to RSI \cref{eq:RSI-smooth} and PL inequality \cref{eq:PL-smooth}, the proof of \cref{proposition:GD-linear-convergence} is very elegant, taking only a few lines. For self-completeness, we provide a simple proof and some additional discussions in \cref{appendix:proof-GD}. In particular, the RSI \cref{eq:RSI-smooth} leads to a quick proof of \cref{eq:linear-convergence-GD-iterates}, and the PL \cref{eq:PL-smooth} allows for a simple proof of \cref{eq:linear-convergence-GD-costs}. It is known that for $L$-smooth convex functions, the two conditions RSI \cref{eq:RSI-smooth} and PL \cref{eq:PL-smooth} are equivalent (cf. \cite[Theorem 2]{karimi2016linear}). Some recent studies, such as \cite{bolte2017error,necoara2019linear,zhang2017restricted}, have explored the relationships among different regularity conditions for linear convergences. A nice summary appeared in \cite[Theorem 2]{karimi2016linear}, but it only works in the context of $L$-smooth functions. 

In this paper, we aim to characterize the relationships among different regularity conditions for \textit{nonsmooth and nonconvex} functions (\Cref{section:equivalent-conditions}), and apply them to derive \textit{simple and clean} proofs for linear convergences of (inexact) proximal point methods for convex and weakly convex (possibly nonsmooth) optimization (\Cref{section:sublinear-linear-PPM,section:inexact-PPM}). Before that, we next introduce some key techniques for analyzing nonsmooth problems. 

\subsection{The notion of slope and Ekeland's variational principle}

In this subsection, we review three important techniques in nonsmooth optimization, 1) the notion of slope, 2) a key distance estimation result, and~3)~Ekeland's variational principle, which will play an important role in our proofs~later. 

We first introduce the Fr\'echet subdifferential. Let $f:\RR^n \to \overline{\RR}$ be a proper closed function. We define the Fr\'echet subdifferential (see e.g., \cite[Page 27]{li2020understanding}): 
$$
\fpartial f(x) := \left \{ s \in \RR^n \mid \liminf_{y \to x }  \frac{ f(y) - f(x) -  \innerproduct{s}{y-x}}{\|y-x\|} \geq 0 \right\}.
$$ 
If $f$ is convex, the Fr\'echet subdifferential $\fpartial f$ is the same as the usual convex subdifferential, i.e., 
$$\fpartial f(x) = \partial f(x) :=  \{s \in \RR^n \mid f(y) \geq f(x) + \innerproduct{s}{y-x}, \forall y \in \RR^n \}, \;\; \forall x \in \RR^n,$$ and if $f$ is differentiable, Fr\'echet subdifferential reduces to the usual gradient, i.e., $ \fpartial f(x) = \{\nabla f(x)\}, \forall x \in \RR^n$. 

\begin{definition}[Slope {\cite[Section 2]{drusvyatskiy2021nonsmooth}}]
\label{definition:slope}
Consider a closed function $f: \RR^n \to \overline{\RR}$ and a point $\bar{x}$ with $f(\bar{x})$ finite. The \textit{slope} of $f$ at $\bar{x}$ is its maximal instantaneous rate of decrease:
\begin{align*}
    |\nabla f|(\bar{x}) := \limsup_{x \to \bar{x}} \frac{ (f(\bar{x})- f(x))^{+}} {\|x - \bar{x}\|},
\end{align*}
where we use the notation $r^{+} = \max\{0,r\}$.
\end{definition}
If $f$ is smooth, the slope $|\nabla f|(\bar{x})$ is simply the norm of its gradient $\|\nabla f(\bar{x})\|$. If $f(x)$ is convex, the slope $|\nabla f|(\bar{x})$ is equivalent to the length of the minimal norm element in the subdifferential, i.e., $|\nabla f|(\bar{x})=\Dist(0, \partial f(\bar{x}))$. If it is $\rho$-weakly convex (i.e., $f + \frac{\rho}{2}\|\cdot\|^2$ is convex), the slope $|\nabla f|(\bar{x})$ is the same as the length of the minimal norm element in the Fr\'echet subdifferential, i.e., $|\nabla f|(\bar{x}) = \Dist(0, \fpartial f(\bar{x}))$ (see 
\cite{ioffe2017variational,drusvyatskiy2013slope} for more details). 
We summarize these useful properties into a lemma below; a brief proof is presented in \Cref{appendix:proof-lemma-2.1} for self-completeness:

\begin{lemma}
    \label{lemma:slope-subdifferential}
    Consider a closed function $f: \RR^n \to \overline{\RR}$ and a point $\bar{x} \in \RR^n$ with $f(\bar{x})$ finite. The following statements for the slope hold.
    \begin{enumerate}
        \item If $f$ is smooth, then $|\nabla f|(\bar{x}) =  \|\nabla f(\bar{x})\|$; 
        \item If $f$ is convex, then $|\nabla f|(\bar{x}) = \Dist(0,\partial f(\bar{x}))$;
        \item If $f$ is $\rho$-weakly convex, then $|\nabla f|(\bar{x}) = \Dist(0,\fpartial f(\bar{x}))$.
    \end{enumerate}
\end{lemma}

We then introduce a key result from \cite[Lemma 2.5]{drusvyatskiy2015curves} that will serve as a crucial step to estimate the distance to the optimal solution set.
\begin{lemma}[Lemma 2.5 \cite{drusvyatskiy2015curves}]
    \label{thm:basic-lemma}
    Let $f: \RR^n \to \overline{\RR}$ be a proper closed function. Suppose for some point $x \in \text{dom}(f)$, there are constants $\alpha < f(x)$ and $r>0, K > 0$ such that $f(x) - \alpha < Kr$ and
    \begin{align*}
        |\nabla f|(u) \geq r, \qquad  \forall u \in [\alpha < f \leq f(x)], \|u-x\| \leq K.
    \end{align*}
    Then the sublevel set $[f\leq \alpha]$ is nonempty and $\Dist(x,[f\leq \alpha]) \leq (f(x)-\alpha)/r.$
\end{lemma}
Note that if $\alpha = f^\star = \inf_{x \in \RR^n} f(x)$ in \cref{thm:basic-lemma}, we can upper bound the distance to the solution set $\Dist(x, [f= f^\star])$ by $(f(x) - f^\star)/r$.

Finally, we review Ekeland's variational principle \cite{ekeland1974variational}. 
\begin{theorem}[\textbf{Ekeland's variational principle} \cite{ekeland1974variational}]
    \label{theorem:Ekeland}
    Let $f: \RR^n \to \overline{\RR} $ be a proper closed function. Suppose $\epsilon > 0$ and $z \in [f \leq  \inf_{x\in \RR^n} f(x) + \epsilon]$.
    Then for any $\rho > 0$ there exists $y \in \RR^n$ such that 
     \begin{subequations}
         \begin{align}
             \|z-y\| &\leq \epsilon/\rho, \label{eq:Ekeland-1} \\
             f(y)  &\leq f(z), \label{eq:Ekeland-2}\\
              f(x) + \rho \|x-y\|  &> f(y), \;\;\forall x \in \RR^n/\{y\}.\label{eq:Ekeland-3}
         \end{align}
     \end{subequations}
\end{theorem}
Note that \cref{eq:Ekeland-1} means that the distance between $y$ and $z$ is bounded, \cref{eq:Ekeland-2} tells us that $y$ is also in the sublevel set, and \cref{eq:Ekeland-3} shows that $y$ uniquely minimizes the function $f+ \rho \|\cdot -y\|$. 

From the definition of slope (\Cref{definition:slope}), one immediate consequence of \cref{theorem:Ekeland} is that the slope of $y$ is also bounded by $\rho$ as
\begin{equation}
    \label{eq:bound-slope}
    \begin{aligned}
        |\nabla f|(y) = \limsup_{x \to y} \frac{ (f(y)- f(x))^{+}} {\|x-y\|} \leq \limsup_{x \to y} \frac{ \rho \cdot \Dist(x,y) } {\|x-y\|} = \rho.
    \end{aligned}
\end{equation}
\Cref{thm:basic-lemma,theorem:Ekeland} will play an important role in our analysis of different regularity conditions in \cref{prop:QG-EB-PL}.

\section{Relationships between regularity conditions}
\label{section:equivalent-conditions}
In this section, we aim to show the equivalence between different regularity conditions in the class of \textit{weakly convex (potentially nonsmooth)} functions. A function $f: \RR^n \to \overline{\RR}$ is called $\rho$-weakly convex if the function $f + \frac{\rho}{2}\|\cdot\|^2$ is convex. The class of weakly convex functions is very~rich: it includes all convex functions, $L$-smooth functions, certain compositions of convex functions with smooth functions, and many cost functions in modern machine learning applications; we refer interested readers to \cite{drusvyatskiy2020subgradient,atenas2023unified} for more details. 
\subsection{Regularity conditions and their relationship}

Let $f:\RR^n \to \overline{\RR}$ be a proper, closed, $\rho$-weakly convex function. Let $S$ be the optimal solution set of $f$, i.e., $S = \argmin_{x \in \RR^n} f(x)$, and we assume $S \neq \emptyset$. Let $f^\star = \min_{x \in \RR^n} f(x)$ and $\nu > 0$. We consider the following five regularity conditions:
\begin{enumerate}
\setlength{\itemsep}{0pt}
    \item \textbf{Local Strong Convexity (SC)}: there exists a positive constant $\mus > 0$ such that 
    \begin{equation} 
    \begin{aligned}
            \label{eq:SC}
        f(x) + \langle g, y - x \rangle + \frac{\mus}{2}
        \|y - x\|^2 & \leq f(y), \quad \forall x, y \in [f \leq f^\star + \nu] , g \in \fpartial f(x). 
    \end{aligned}
     \tag{SC}
    \end{equation}

    \item  \textbf{Restricted Secant Inequality (RSI)}: there exists a positive constant $\mur > 0$ such that
    \begin{equation}
    \begin{aligned}
        \label{eq:RSI}
        \mur \cdot  \Dist^2(x,S) &  \leq \innerproduct{g}{x- \hat{x}} , \quad \forall x \in [f  \leq  f^\star + \nu], g \in \fpartial f(x), \hat{x} \in \Pi_{S}(x).
        \end{aligned}
        \tag{RSI}
    \end{equation}
    
    \item \textbf{Error bound (EB)}\footnote{Error bound is closely related to the notion of \textit{metric subregularity} at an optimal point $x^\star$ for 
    $0 \in \fpartial f(x^\star)$ \cite[Definition 2.3]{artacho2008characterization}: there exist a constant $\alpha>0$ and a set $\mathcal{U}$ containing $x^\star$ such that $\Dist(x,(\fpartial f)^{-1}(0)) \leq \alpha \cdot \Dist(0,\fpartial f(x)), \forall x \in \mathcal{U}$.}: 
    there exists a constant $\mue >0$ such that
    \begin{equation}
        \label{eq:EB}
         \Dist (x,S) \leq \mue  \cdot  \Dist(0,\fpartial f(x)), \quad \forall x \in [f \leq f^\star + \nu].
         \tag{EB}
    \end{equation}

    \item \textbf{Polyak-Łojasiewicz (PL) inequality}\footnote{To be consistent with the smooth case in \cite{karimi2016linear}, we call the property \cref{eq:PL} Polyak-Łojasiewicz, which is usually used for smooth functions. The property \cref{eq:PL} is actually a special case of the \textit{Kurdyka-Łojasiewicz inequality} $\varphi^\prime (f(x)-f^\star)\Dist(0,\fpartial f(x))\geq 1 $ with  $\varphi (s) = c s^{1/2}$ and $c > 0$.}: there exists a constant $\mup > 0$ such that 
    \begin{equation}\label{eq:PL}
    2\mup \cdot (f(x) - f^\star) \leq \Dist^2(0,\fpartial f(x)),\quad \forall x \in [ f \leq f^\star + \nu]. 
    \tag{PL}
    \end{equation}

    \item \textbf{Quadratic Growth (QG)}: there exists a constant $\muq > 0$ such that
    \begin{equation}
        \label{eq:QG}
         \frac{\muq}{2} \cdot \Dist^2(x,S) \leq f(x) - f^\star, \quad \forall x \in [f \leq f^\star + \nu].
        \tag{QG}
    \end{equation}
   \end{enumerate}
All the regularity conditions above are defined over a sublevel set $[f \leq f^\star + \nu]$. If~$v = +\infty$, then they are global. In particular, \cref{eq:SC} imposes a quadratic lower bound at every point in the sublevel set. 
On the other hand, \cref{eq:RSI}, \cref{eq:EB} and \cref{eq:PL} all require a certain growth of the minimal norm element in $\fpartial f(x)$ when moving away from its solution set $S$ or optimal value $f^\star$. It is easy to see that \cref{eq:RSI}, \cref{eq:EB}  and \cref{eq:PL} all imply that every stationary point $0 \in \fpartial f(x)$ in the sublevel set $[f\leq f^\star + \nu]$ is a global minimum (but they do not imply the uniqueness of stationary points). Finally, \cref{eq:QG} shows that $f(x)$ grows at least quadratically when moving away from the solution set $S$. Unlike \cref{eq:RSI}, \cref{eq:EB}  and \cref{eq:PL}, this \cref{eq:QG} alone allows suboptimal stationary points. 

Our first technical result summarizes the relationships among the five regularity conditions. 
\begin{theorem}
    \label{prop:QG-EB-PL}
    Let $f:\RR^n \to \overline{\RR}$ be a proper closed $\rho$-weakly convex function with $f^\star = \min_{x \in \RR^n}f(x)$ and $S = \argmin_{x \in \RR^n}f(x)$. Suppose $S \neq \emptyset$ and let $\nu >0$ be the same constant throughout \cref{eq:SC}, \cref{eq:RSI}, \cref{eq:EB}, \cref{eq:PL}, and \cref{eq:QG}. The following relationship holds 
   \begin{equation} \label{eq:relationship-weakly-convex}
    \cref{eq:SC} 
    \rightarrow  \cref{eq:RSI} \rightarrow
    \cref{eq:EB}   \equiv \cref{eq:PL} \rightarrow \cref{eq:QG}.
   \end{equation}
    Furthermore, if the \cref{eq:QG} coefficient satisfies $\muq > \rho$ (including the case that $f$ is convex),   
   then the following equivalence holds
   \begin{equation} \label{eq:relationship-convex}
        \cref{eq:RSI}\equiv \cref{eq:EB}   \equiv \cref{eq:PL}  \equiv \cref{eq:QG}.
   \end{equation}
\end{theorem}

\begin{table}[t]
\setlength{\abovecaptionskip}{0pt}
\setlength{\belowcaptionskip}{0pt}
    \caption{Constant conversion for \cref{eq:RSI}, \cref{eq:EB}, \cref{eq:PL} and \cref{eq:QG}, under the condition $\muq > \rho$ in \cref{prop:QG-EB-PL}. Each row represents the corresponding coefficients for other conditions while holding one condition constant. For example, the second row shows the coefficients for \cref{eq:EB}, \cref{eq:PL}, and \cref{eq:QG} assuming that  \cref{eq:RSI} holds with a positive constant $\mur$.  }
    \label{table:equivalent-condtions}
    \renewcommand*{\arraystretch}{1.75}
  {\centering
    \begin{tabular}{|c|c|c|c|c|}
   \hline
    &\cref{eq:RSI} & \cref{eq:EB}& \cref{eq:PL}  & \cref{eq:QG}\\
     \hline
    \cref{eq:RSI} & $\mur$ & $[\frac{1}{\mur},+\infty)$ & $(0,\frac{\mur^2}{2\mur + \rho}]$ & $(0,\frac{\mur}{2}]$\\
    \hline
     \cref{eq:EB} & $(0,\frac{1-2\mue \rho}{4 \mue}]$ & $\mue$ & $(0,\frac{1}{2\mue + \rho \mue^2}]$  & $(0,\frac{1}{2\mue}]$ \\
     \hline
    \cref{eq:PL} & $(0,\frac{\mup - 2 \rho}{4}]$ & $[\frac{1}{\mup},+\infty)$& $\mup$ & $(0,\frac{\mup}{2}]$ \\
    \hline
    \cref{eq:QG} & $(0,\frac{\muq - \rho}{2}]$  & $[\frac{2}{\muq - \rho},+\infty)$& $(0,\frac{(\muq - \rho)^2}{4 \muq}]$ & $\muq$\\
   \hline
    \end{tabular}
    }
\end{table}

Note that if a function $f$ is convex and satisfies \cref{eq:QG} with $\muq > 0$, then \cref{eq:relationship-convex} holds naturally as $f$ is naturally $0$-weakly convex.~\cref{prop:QG-EB-PL} includes \cite[Theorem 2]{karimi2016linear} as a special case, in which only $L$-smooth functions are considered. It is easy to see that all~$L$-smooth functions are also $L$-weakly convex. 
Even for differentiable functions, \cref{prop:QG-EB-PL} is more general than \cite[Theorem 2]{karimi2016linear} in the sense that 1) we require no Lipschitz constant $L$ for gradients to ensure the equivalency among \cref{eq:EB}, \cref{eq:PL}, and \cref{eq:QG} in the convex case; 2) the condition $\muq > \rho$ is new and does not mean that $f$ is convex. We will provide further illustrative examples in  \cref{subsection:examples}. Under the condition $\muq > \rho$, we can further infer the coefficients for different regularity conditions from each other; the details are listed in \cref{table:equivalent-condtions}.

The proof of \cref{prop:QG-EB-PL} relies heavily on the notion of \textit{slope} in \Cref{definition:slope}, Ekeland’s variational principle in \cref{theorem:Ekeland}, and the technical result in \cref{thm:basic-lemma}. We will present the proof details in \Cref{subsection:proof-equivalence}. 
Alternative proofs based on \textit{subgradient flows} \cite{bolte2017error} are also possible. Indeed, one key step in the proof of \cite[Theorem 2]{karimi2016linear} is based on \textit{gradient flows} for smooth functions, which is a special case of \textit{subgradient flows} for nonsmooth cases.

\begin{remark} \label{remark:connections}
The regularity conditions \cref{eq:EB}, \cref{eq:QG} and \cref{eq:PL} have been discussed for different function classes in the literature. For the smooth case, we refer to \cite[Theorem 2]{karimi2016linear} for a nice summary; also see \cite{rebjock2023fast,guille2021study} for related discussions. For nonsmooth convex functions, the equivalence between \cref{eq:EB} and \cref{eq:QG} has been recognized in \cite[Theorem 3.3]{drusvyatskiy2018error} and \cite[Theorem 3.3]{artacho2008characterization}, and the equivalence between \cref{eq:PL} and \cref{eq:QG} is established in \cite[Theorem 5]{bolte2017error}. Thus, \cref{eq:EB}, \cref{eq:PL}, and \cref{eq:QG} are equivalent for the class of nonsmooth convex~functions \cite[Proposition 2]{ye2021variational}; see also \cite{zhu2023unified} for a recent discussion. 
Our \cref{prop:QG-EB-PL} extends these results to $\rho$-weakly convex functions. The most closely related work is \cite{drusvyatskiy2021nonsmooth} which focuses on nonsmooth optimization using Taylor-like models. Indeed, we specialize the proof of \cite[Theorem 3.7, Proposition 3.8, and Corollary 5.7]{drusvyatskiy2021nonsmooth} into our setting and prove the relationship in \cref{eq:relationship-convex,eq:relationship-weakly-convex} directly using the slope technique. 
We note that the implication from \cref{eq:QG} to \cref{eq:EB}/\cref{eq:PL} is not true in general. Yet, with the condition $\muq > \rho$ in \cref{prop:QG-EB-PL}, all four regularity conditions \cref{eq:RSI}, \cref{eq:EB},  \cref{eq:PL} and \cref{eq:QG} are equivalent.
\hfill $\square$
\end{remark}

\subsection{Examples}
\label{subsection:examples}

In principle, all the five properties in \cref{prop:QG-EB-PL} are generalizations of quadratic functions to non-quadratic, nonconvex, and even nonsmooth cases. For illustration, let us first consider the simplest quadratic function $f(x) = x^2$, which is convex and differentiable. It is clear that  $\fpartial f(x) = \{2x\}$ and $S = \{0\}$. It is also immediate to verify that \cref{eq:SC} holds with $0 < \mus \leq 2$, \cref{eq:RSI} holds with $0 < \mur \leq  2$, \cref{eq:EB} holds with $\mue  \geq  1/2$, \cref{eq:PL} holds with $0<\mup \leq 2$, and \cref{eq:QG} holds with $0<\muq \leq 2$. 
Consider another simple convex function $f(x) = x^2, \text{if}~|x|\leq 1,$ and $f(x) =\frac{1}{2}x^4 + \frac{1}{2}$ otherwise. All the five properties hold for this function, but it is not $L$-smooth globally. 

Let us move away from convex functions, and consider $f(x) = x^2 + 6 \sin^2(x)$. It is clear that this function satisfies \cref{eq:QG} globally, however, there exist suboptimal stationary points and consequently \cref{eq:EB} and \cref{eq:PL} do not hold globally. Thus, in this case (with $\nu = +\infty$), the relationship \cref{eq:relationship-weakly-convex} is strict, and \cref{eq:QG} is more general than the other conditions. 

Finally, we consider a $\rho$-weakly convex function with a QG constant satisfying $\muq > \rho$ in the following example. \Cref{prop:QG-EB-PL} confirms that this function will also satisfy \Cref{eq:RSI}, \Cref{eq:PL}, and \Cref{eq:EB}.  
\begin{example}\label{example:weakly-convex}
     Consider the function 
    \begin{equation*}
        \begin{aligned}
            f(x) = \begin{cases}
             -x^2+1 & \text{if}~ -1  <  x < -0.5, \\
             3(x+1)^2 & \text{otherwise}.
        \end{cases}
        \end{aligned}
    \end{equation*}
    It can be verified that the function is not convex (due to the part $-1<x<-0.5$) but $2$-weakly convex with a unique minimizer $x^\star = -1$ and the minimum value $f^\star = 0$. See \Cref{figure:weakly-convex-muq-greater-rho} for an illustration of these properties. The Fr\'echet subdifferential of this function can be calculated as 
    \begin{equation*}
        \begin{aligned}
            \fpartial f(x) = \begin{cases}
             -2x & \text{if}~ -1  <  x < -0.5, \\
             6(x+1) & \text{if}~ x < -1~\text{or}~x>-0.5, \\
             [0,2]    & \text{if}~x = -1, \\
             [1,3]    & \text{if}~x = -0.5.
             \end{cases} 
        \end{aligned}
    \end{equation*}
    The length of the minimal element is then
    \begin{equation*}
        \begin{aligned}
            \Dist(0,\fpartial f(x)) = \begin{cases}
             |2x| & \text{if}~-1<x<-0.5, \\
             |6(x+1)| & \text{if}~ x < -1~\text{or}~x>-0.5, \\
             0    & \text{if}~x = -1, \\
             1    & \text{if}~x = -0.5.
             \end{cases}
        \end{aligned}
    \end{equation*}
    As expected from \Cref{prop:QG-EB-PL}, we can see that \cref{eq:QG} holds for $0< \muq\leq 6$,  \cref{eq:EB} holds for $\mue \geq 1/2$, and \cref{eq:PL} holds for $0 < \mup \leq 2/3$. These can also be observed in \cref{figure:weakly-convex-muq-greater-rho}.
\end{example}

\begin{figure}[t]
    \centering
    \setlength{\abovecaptionskip}{0pt}
    \includegraphics[width=\textwidth]{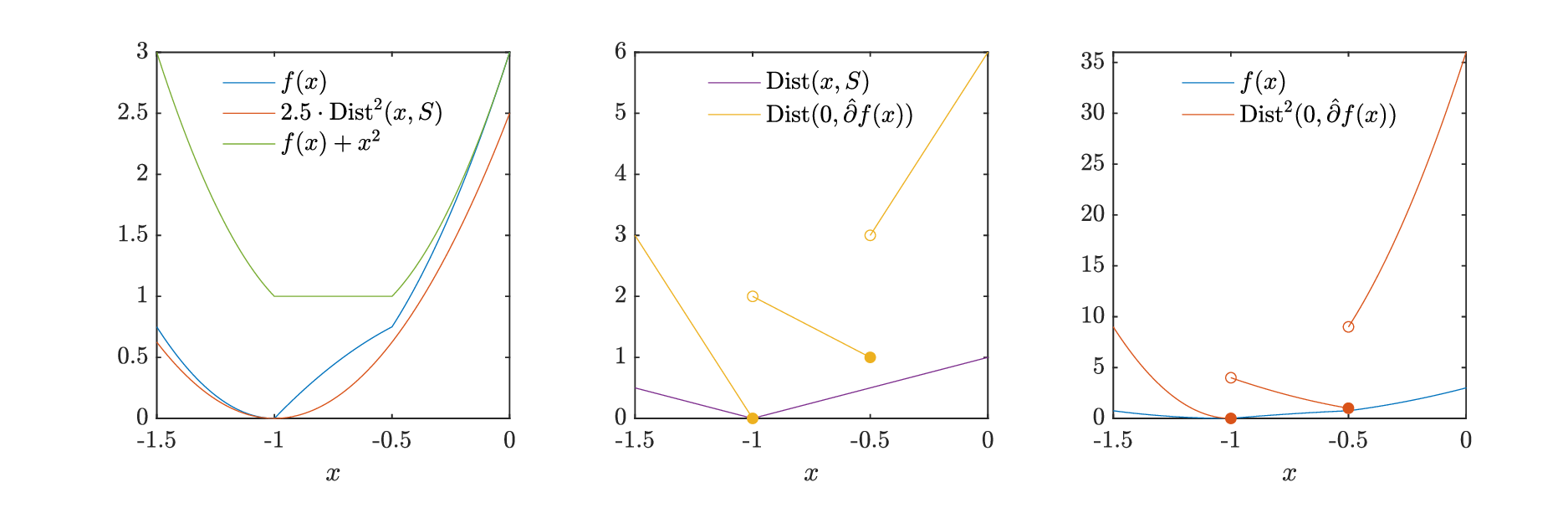}
    \caption{A nonconvex function with $f^\star = 0$ that satisfies the equivalency $\cref{eq:RSI}\equiv \cref{eq:EB} \equiv \cref{eq:PL} \equiv \cref{eq:QG}$. Left: the function $f$ (blue line) is $2$-weakly convex (confirmed by the green line), and also satisfies \cref{eq:QG} with $\muq = 5$ (confirmed by the red line);  Middle: $f$ satisfies \cref{eq:EB} as $\Dist(0, \fpartial f(x))$ (yellow curve) is lower bounded by $\Dist(x,S)$ (purple curve). Right: $f$ satisfies \cref{eq:PL} as $\Dist^2(0,\fpartial f(x))$ (red curve) is lower bounded by $f(x)$ (blue curve).}
    \label{figure:weakly-convex-muq-greater-rho}
\end{figure}

\subsection{Proof of \Cref{prop:QG-EB-PL}} \label{subsection:proof-equivalence}
The directions \cref{eq:SC} $\Rightarrow$ \cref{eq:RSI}, \cref{eq:RSI} $\Rightarrow$ \cref{eq:EB}, \cref{eq:EB} $\Rightarrow$   \cref{eq:PL}, and \cref{eq:QG} with  $\muq > \rho$ $\Rightarrow$ \cref{eq:RSI} follow easily from some elementary algebras. The directions \cref{eq:PL} $\Rightarrow$ \cref{eq:EB} and \cref{eq:EB} $\Rightarrow$ \cref{eq:PL} require the equivalence in \cref{lemma:slope-subdifferential}. Most directions are not difficult to prove, except that the direction \cref{eq:PL} $\Rightarrow$ \cref{eq:EB} requires some sophisticated arguments involving \cref{thm:basic-lemma,theorem:Ekeland}. To start the proof, we first recall that $S$ denotes the set of minimizers of $f$ and $f^\star = \min_{x\in \RR^n }f(x)$ and we assume $S$ is nonempty. 
\begin{itemize}
    \item \cref{eq:SC} $\Rightarrow$ \cref{eq:RSI}: Suppose $f$ satisfies \cref{eq:SC} with constant $\mus>0$. For any $x \in [f\leq f^\star + \nu]$ and $g \in \fpartial f(x)$, we let $\hat{x} \in \Pi_{S}(x)$. It follows that 
    \begin{equation*}
        \begin{aligned}
            & f(x) + \innerproduct{g}{\hat{x}-x} + \frac{\mus}{2} \cdot \Dist^2(x,S) \leq f(\hat{x}) \\
           \Longrightarrow ~ & \frac{\mus}{2} \cdot \Dist^2(x,S) \leq f(\hat{x}) - f(x) + \innerproduct{g}{x-\hat{x}} \\
           &\qquad \qquad \qquad \; \leq \innerproduct{g}{x-\hat{x}},
        \end{aligned}
    \end{equation*}
    where the last inequality comes from the fact that $f(\hat{x}) - f(x) \leq 0$.
    
    \item \cref{eq:RSI} $\Rightarrow$ \cref{eq:EB}: Suppose $f$ satisfies \cref{eq:RSI} with constant $\mur > 0$. Let $x \in [f\leq f^\star + \nu]$, $g$ be the minimal norm element in $\fpartial f(x)$, and $\hat{x} \in \Pi_{S}(x)$. Then, by definition of \cref{eq:RSI}, we have 
    \begin{equation*}
        \begin{aligned}
            \innerproduct{g}{x - \hat{x}} \geq \mur \cdot \Dist^2(x,S).
        \end{aligned}
    \end{equation*}
    Applying Cauchy-Schwarz on the left side yields 
    \begin{equation*}
        \begin{aligned}
            \Dist(0,\fpartial f (x))  \geq \mur \cdot \Dist(x,S).
        \end{aligned}
    \end{equation*}

     \item  \cref{eq:EB} $\Rightarrow$   \cref{eq:PL}: Suppose $f$ satisfies \cref{eq:EB} with constant $\mue>0$. Let $x \in [f \leq f^\star + \nu]$, $\hat{x} \in \Pi_{S}(x)$, and $g$ be the minimal norm element in $\fpartial f(x)$. From the subdifferential property for $\rho$-weakly convex function (i.e. \cref{eq:subdifferential-inequality-weakly-convex} in \cref{appendix:background}), we have 
    \begin{align*}
        f^\star \geq f(x) + \innerproduct{g}{\hat{x} - x} - \frac{\rho}{2}\|\hat{x} - x\|^2.
    \end{align*}
    Using the triangle inequality, it follows that
    \begin{align*}
        f(x) - f^\star & \leq \Dist(0,\fpartial f(x))\Dist(x,S) + \frac{\rho}{2}\Dist^2(x,S) \\
        & \leq \mue \cdot \Dist^2(0,\fpartial f(x)) + \frac{\rho \mue^2}{2}\Dist^2(0,\fpartial f(x)) \\
        & = \left(\frac{2 \mue +\rho \mue^2}{2} \right) \Dist^2(0,\fpartial f(x)).
    \end{align*}
    Dividing both sides by $(2 \mue +\rho \mue^2)/2$ gets the desired constant. 
     \item {\cref{eq:PL} $\Rightarrow$ \cref{eq:EB}:  
     Suppose $f$ satisfies \cref{eq:PL} with constant $\mup > 0$. 
    As $f$ is $\rho$-weakly convex, applying \cref{lemma:slope-subdifferential} on \cref{eq:PL} shows that 
    \begin{equation} 
        \label{eq:slop-PL-f}
        |\nabla f|^2(x) \geq 2\mup \cdot (f(x) - f^\star), \quad \forall x \in [f\leq f^\star + \nu].
    \end{equation}
    Define the function $g$ as $x \mapsto g(x) = (f(x) - f^\star)^{1/2}$. We claim that 
    \begin{equation}
        \label{eq: g_slope_inequality-preview}
        |\nabla g|(x)\geq \sqrt{\frac{\mup}{2}}, \; \forall x  \in [0<g \leq \sqrt{\nu}].
    \end{equation}
With \cref{eq: g_slope_inequality-preview}, we can establish the implication \cref{eq:PL} $\Rightarrow$ \cref{eq:EB} as follows. Fix a point $x \in [ 0 < g \leq \sqrt{\nu}]$ and choose $K > g(x) \sqrt{\frac{2}{\mup}}$. We have $g(x)< K \sqrt{\frac{\mup}{2}}$ and 
    \begin{equation*}
    \begin{aligned}
        |\nabla g|(u) & \geq    \sqrt{\frac{\mup}{2}}, \quad \forall u \in [0 < g \leq g(x)], \|x-u\| \leq K.
    \end{aligned}
    \end{equation*}
    Choosing $r = \sqrt{\frac{\mup}{2}}, \alpha = 0$ in \cref{thm:basic-lemma}, we have the following bounds 
    \begin{equation*}
        \begin{aligned} 
            \Dist(x,S) & = \Dist(x,[g\leq 0]) \\ & \leq \sqrt{\frac{2}{\mup}}g(x) = \sqrt{\frac{2}{\mup}}(f(x)-f^\star)^{1/2} \\
       & \leq \frac{1}{\mup} |\nabla f|(x) = \frac{1}{\mup} \Dist(0,\fpartial f(x)), \forall x\in [f \leq f^\star + \nu],
        \end{aligned}
    \end{equation*}
    where the first inequality is from \cref{thm:basic-lemma}, the second inequality comes from \Cref{eq:slop-PL-f}, and the third equality applies \cref{lemma:slope-subdifferential}.
    This shows the implication \cref{eq:PL} $\Rightarrow$ \cref{eq:EB}.

    \textbf{Proof of the claim \cref{eq: g_slope_inequality-preview}:} 
    Fix any $\Bar{x} \in [0<g \leq \sqrt{\nu}]$. We shall show 
\begin{equation}\label{eq: g_slope_inequality}
    |\nabla g|(\bar{x})\geq \sqrt{\frac{\mup}{2}}.
\end{equation}
Note that if $|\nabla g|(\bar{x}) = +\infty$, then \eqref{eq: g_slope_inequality} holds trivially. Hence, we only prove \eqref{eq: g_slope_inequality} for the case $|\nabla g|(\bar{x}) < +\infty$. By the definition of the slope $|\nabla f|(\bar{x}) = \lim \sup_{y\rightarrow \bar{x}} \frac{(f(\bar{x})-f(y))^+}{\norm{y-\bar{x}}}$ and by our choice of $\bar{x}$ satisfying $|\nabla f|(\bar{x})\geq 
2\mu_p(f(\bar{x})-f^\star)>0$, we claim that there exists a sequence $\{y_i\}$ satisfying the following conditions:
\begin{equation}\label{eq: nabla_g_bound_claim}
\{y_i\}\rightarrow \bar{x}, \quad 
\{f(y_i)\}\uparrow f(\bar{x}), 
\quad |\nabla f|(\bar{x}) = \lim_{i \to  +\infty} \frac{f(\bar{x})-f(y_i)}{\norm{y_i-\bar{x}}},
\end{equation}
where $\{f(y_i) \}\uparrow f(\bar{x})$ means $f(y_i)\leq f(\bar{x})$ for all $i$ and $\{f(y_i)\}\rightarrow f(\bar{x})$. If \cref{eq: nabla_g_bound_claim} holds, then \eqref{eq: g_slope_inequality} can be readily established using basic calculus. Indeed, using the Taylor expansion of the function $x\mapsto \sqrt{x}$ near $f(\bar{x})-f^\star$, we have 
\begin{align*}
|\nabla g|(\bar{x}) & = \limsup_{y\to \bar{x}} \frac{ \left ((f(\bar{x}) - f^\star)^{\frac{1}{2}} - (f(y) - f^*)^{\frac{1}{2}} \right )^+}{\|y - \bar{x}\|}\\
& \geq 
\limsup_{i\rightarrow +\infty} \frac{(f(\bar{x}) - f^*)^{\frac{1}{2}}-(f(y_{i}) - f^*)^\frac{1}{2}}{\norm{y_{i}-\bar{x}}}\\
& = \limsup_{i\rightarrow +\infty} \frac{
\frac{1}{2 (f(\bar{x}) - f^*)^{\frac{1}{2}}}(f(\bar{x})-f(y_i)) + o(|f(y_i)-f(\bar{x})|) 
}{\norm{y_{i}-\bar{x}}}\\
& \overset{(a)}{=} 
\frac{|\nabla f|(\bar{x})}{
2 (f(\bar{x}) - f^*)^{\frac{1}{2}}} \\
& \overset{(b)}{\geq} \sqrt{\frac{\mup}{2}},
\end{align*}
where $(a)$ is due to \eqref{eq: nabla_g_bound_claim}, and $(b)$ applies \eqref{eq:slop-PL-f}. It remains to prove the claim \cref{eq: nabla_g_bound_claim}. 

\textbf{Proof of the claim \cref{eq: nabla_g_bound_claim}:} The conditions $\{y_i\}\rightarrow \bar{x}$, $f(y_i) \leq f(\bar{x}) $ for all $i$, and $|\nabla f|(\bar{x}) = \lim_{i \to  +\infty} \frac{f(\bar{x})-f(y_i)}{\norm{y_i-\bar{x}}}$ follow from the definition of $|\nabla f |(\bar{x})$ and the assumption  $|\nabla f |(\bar{x}) > 0$. Below, we further establish $\{f(y_i) \}\uparrow f(\bar{x})$. Suppose for contradiction that the sequence $\{f(y_i) \}$ does not converge to $f(\bar{x})$. Then, let $\epsilon = \frac{1}{2} (f(\bar{x}) -  f^\star) > 0$, there exists a subsequence $\{f(y_{i_k})\}_k$ such that $ \epsilon \leq f(\bar{x}) - f(y_{i_k})$. However, this subsequence will result in
\begin{align*}
    |\nabla g(\bar{x})| &\geq \limsup_{k\rightarrow +\infty} \frac{ (f(\bar{x}) - f^*)^{\frac{1}{2}}-(f(y_{i_k}) - f^*)^\frac{1}{2}}{\norm{y_{i_k}-\bar{x}}} \\ & \geq \limsup_{k\rightarrow +\infty}     
\frac{(f(\bar{x}) - f^*)^{\frac{1}{2}}-(f(\bar{x}) - f^*-\epsilon)^{\frac{1}{2}} }{\norm{y_{i_k}-\bar{x}}} \\
& = \limsup_{k\rightarrow +\infty}     
\frac{(f(\bar{x}) - f^*)^{\frac{1}{2}}-( \frac{1}{2}(f(\bar{x}) - f^*))^{\frac{1}{2}} }{\norm{y_{i_k}-\bar{x}}} \\
& = +\infty,
\end{align*}
violating the assumption $|\nabla g|(\bar{x}) < +\infty$. This proves the claim \cref{eq: nabla_g_bound_claim}. 
}

    \item \cref{eq:EB} $\Rightarrow$ \cref{eq:QG}: 
    Suppose $f$ satisfies \cref{eq:EB} with $\mue > 0$. Fixing any $x \in [f^\star<f\leq f^\star + \nu]$, we have $x \in [f^\star<f \leq f^\star + \epsilon]$ with $\epsilon = f(x) - f^\star$. Choosing $\rho = \sqrt{\epsilon/\mue}$, \Cref{theorem:Ekeland,eq:bound-slope} ensure that there exists a $y$ such that $y \in [f\leq f^\star+ \epsilon]$ and $|\nabla f|(y) \leq \rho$.
    Thus,
    \begin{align*}
         \mue \rho \geq \mue \cdot |\nabla f|(y)& \overset{(a)}{\geq} \Dist(y,S)\\
         &\overset{(b)}{\geq} \Dist(x,S) - \Dist(x,y) \\
         &\overset{(c)}{\geq} \Dist(x,S) - \epsilon/\rho,
    \end{align*}
    where $(a)$ uses \cref{lemma:slope-subdifferential} (i.e., we have $|\nabla f|(y) = \Dist(0,\fpartial f(y))$ for weakly convex functions) and \cref{eq:EB}, $(b)$ applies the triangle inequality for a point to a set, and $(c)$ comes from \cref{eq:Ekeland-1}.
    Substituting $\epsilon = f(x) - f^\star$ and $\rho = \sqrt{\frac{\epsilon}{\mue}}$ and rearranging the above inequality yields 
    \begin{align*}
        2\sqrt{\mue} (f(x) - f^\star)^{1/2} \geq \Dist(x,S). 
    \end{align*}
    Squaring both sides completes the proof of \cref{eq:EB} $\Rightarrow$ \cref{eq:QG}.
\end{itemize}

Summarizing the relationship above leads to the first part of \Cref{prop:QG-EB-PL} in \Cref{eq:relationship-weakly-convex}. We now prove the second part in \Cref{eq:relationship-convex}. 

\begin{itemize}
    \item \cref{eq:QG} with  $\muq > \rho$ $\Rightarrow$ \cref{eq:RSI}: Let $x  \in [f \leq f^\star + \nu]$, $\hat{x} \in \Pi_{S}(x)$, and $ g \in \fpartial f(x)$. From the assumption of \cref{eq:QG} and the property of subgradient of weakly convex function in \cref{eq:subdifferential-inequality-weakly-convex}, we have 
    \begin{align*}
        \frac{\muq}{2} \cdot \Dist^2(x,S) \leq  f(x) - f^\star & \leq   \innerproduct{g}{x - \hat{x} } + \frac{\rho}{2} \Dist^2(x,S).
    \end{align*}
    Rearranging terms yields
    \begin{equation*}
        \begin{aligned}
            \left (\frac{\muq - \rho}{2} \right)  \cdot \Dist^2(x,S) \leq \innerproduct{g}{x - \hat{x}}.
        \end{aligned}
    \end{equation*}
    \end{itemize}
This completes our proof of \Cref{prop:QG-EB-PL}.

\section{Proximal point method for convex optimization}
\label{section:sublinear-linear-PPM}

In this section, we will utilize the regularity conditions in \Cref{section:equivalent-conditions} to derive fast linear convergence guarantees of the classical proximal point method (PPM) \cite{rockafellar1976augmented} for convex (potentially nonsmooth) optimization and generalize it to weakly convex settings. PPM is a conceptually simple algorithm, which has been historically used for guiding algorithm design and analysis, such as proximal bundle methods \cite{lemarechal1981bundle} and augmented Lagrangian methods \cite{rockafellar1976augmented}. It has recently found increasing applications in modern machine learning; see \cite{drusvyatskiy2017proximal}. 

\subsection{Proximal point method}

Consider the optimization problem  \vspace{-1mm}
\begin{equation}
\label{eq:general-minimization}
    f^\star = \min_x \; f(x), \vspace{-1mm}
\end{equation}
where $f: \RR^n \to \overline{\RR}$ is a proper closed convex function.  Note that \cref{eq:general-minimization} is also~an~abstract model for constrained optimization, since given a closed convex set $X$, we can define $\bar{f}(x) = f(x)$ if $x \in X$, and $\bar{f}(x) = +\infty$ otherwise. Let $S = \argmin_x \, f(x)$.  
We define the \textit{proximal mapping} as
\begin{equation} \label{eq:PPM-subproblem}
    \text{prox}_{\alpha f}(\xk):=\argmin_{x \in \mathbb{R}^n}\; f(x) + \frac{1}{2\alpha} \left\|x - \xk\right\|^2,
\end{equation}
where $\alpha > 0$. 
Starting with any initial point $x_0$, the PPM generates a sequence of points as follows
\begin{equation}
    \label{eq:PPM}
     \xknext =   \mathrm{prox}_{c_k f}(\xk), \quad k = 0, 1, 2, \ldots
\end{equation}
where $\{c_k\}_{k\geq 0}$ is a sequence of positive real numbers. The quadratic term in \cref{eq:PPM-subproblem} makes the objective function strongly convex and always admits a unique solution. The iterates \cref{eq:PPM} are thus well-defined. 

The convergence of PPM \cref{eq:PPM} for (nonsmooth) convex optimization has been studied since the 1970s \cite{rockafellar1976monotone}. The sublinear convergence is relatively easy to establish, and many different assumptions exist for linear convergences of \cref{eq:PPM}; see \cite{rockafellar1976monotone,luque1984asymptotic,leventhal2009metric,cui2016asymptotic,drusvyatskiy2018error}. However, as we will highlight later, some assumptions are restrictive and the corresponding proofs are sophisticated and nontransparent. We aim to provide simple proofs under the general regularity conditions in \Cref{section:equivalent-conditions}.   

\subsection{(Sub)linear convergences of PPM} \label{subsection:sublinear-exact}

Under a very general setup, the PPM \cref{eq:PPM} converges at a sublinear rate for cost value gaps, and the iterates converge asymptotically, as summarized in \cref{theo:PPM-sublinear-convergence}. This result is classical \cite[Theorem 2.1]{guler1991convergence}, and a new bound with an improved constant 4 is also available in \cite[Theorem 4.1]{taylor2017exact} using the performance estimation technique. 

\begin{theorem}[Sublinear convergence {\cite[Theorem 2.1]{guler1991convergence}}] \label{theo:PPM-sublinear-convergence}
     Let $f\!:\! \RR^n \!\to \! \overline{\RR}$ be a proper closed convex function with 
     $S = \argmin_{x \in \RR^n}f(x)$ and $S\neq \emptyset$. Then, the iterates \cref{eq:PPM} with a positive sequence $\{c_k\}_{k\geq 0}$~satisfy
    \begin{equation} \label{eq:PPM-sublinear-convergence-cost}
        f(\xk) - f^\star \leq  \Dist^2(x_0,S)/(2\textstyle \sum_{t=0}^{k-1}c_t).
    \end{equation} 
    If we further have $\lim_{k \to \infty }\sum_{t=0}^{k-1}c_t = \infty$, then the iterates converge to an optimal solution $ \bar{x} $ asymptotically, i.e., $\lim_{k \to \infty} x_k = \bar{x}$, where $\bar{x} \in S$.   
\end{theorem}

The proof of \cref{eq:PPM-sublinear-convergence-cost} is immediate from a telescope sum via the following one-step improvement:
\begin{equation} 
       2c_k( f(\xknext)-  f(x^\star)) \leq   \|\xk - x^\star \|^2 - \|\xknext - x^\star \|^2, \quad \forall c_k >0, x^\star \in S. \label{eq:PPM-distant-nonincreasing}
\end{equation}
This fact is not difficult to establish. Note that choosing any constant step size $c_k = c > 0$ in \cref{eq:PPM-sublinear-convergence-cost} directly implies the common sublinear rate $\mathcal{O}(1/k)$. In \cref{theo:PPM-sublinear-convergence}, $f$ does not need to be $L$-smooth, and it can also be non-differentiable. Thus, the guarantees in \cref{theo:PPM-sublinear-convergence} are much stronger than those by (sub)gradient methods. This is because the proximal mapping \cref{eq:PPM-subproblem} is a stronger oracle than simple (sub)gradient updates. 

Similar to GD in \cref{section:Motivation-Preliminaries},
when the function $f$ satisfies additional regularity conditions, the PPM enjoys linear convergence. Our next main technical result establishes linear convergences of the PPM under the common regularity conditions in \cref{prop:QG-EB-PL}. 

\begin{theorem}[Linear convergence]
    \label{thm-linear}
    Let $f\!:\! \RR^n \!\to \! \overline{\RR}$ be a proper closed convex function with
    $S = \argmin_{x \in \RR^n}f(x)$, $S \neq \emptyset$,~and~$\nu > 0$. 
    Suppose $f$ satisfies \cref{eq:PL} (or  \cref{eq:EB}, \cref{eq:RSI}, \cref{eq:QG}) over the sublevel set $[f \leq f^\star + \nu]$. 
    Then, for all $k \geq k_0 =  \frac{\Dist^2(x_0,S)}{2\nu \inf_{k\geq 0}c_k }$ steps, the iterates \cref{eq:PPM} with a positive sequence $\{c_k\}_{k\geq 0}$ bounded away from zero enjoy linear convergence rates, i.e., 
    \begin{subequations}
    \begin{align}
        f(\xknext) - f^\star & \leq  \omega_k \cdot (f(\xk) - f^\star), \label{eq:cost-gap-linear}\\ 
        \Dist(\xknext,S) & \leq \theta_k \cdot \Dist(\xk,S), \label{eq:iterate-linear}
    \end{align}
    \end{subequations}
    where the constants are  
    \begin{equation*}     
      \omega_k = \frac{1}{1+ c_k \mup}  < 1, \quad \theta_{k} =\min \left \{ \frac{ 1}{\sqrt{c_k\mu_{\mathrm{q}} +1} } , \frac{1}{\sqrt{c_k^2 /\mue^2+ 1}} \right \}< 1.
    \end{equation*}
\end{theorem}

\begin{proof} 
    The sublinear convergence in \cref{theo:PPM-sublinear-convergence} ensures that the iterate $x_k$ reaches $[f\leq f^\star + \nu]$ after at most $k_0$ iterations. Once $\xk$ is within  $[f\leq f^\star + \nu]$, all the properties \cref{eq:EB}, \cref{eq:PL}, \cref{eq:RSI}, and \cref{eq:QG} are equivalent by \cref{prop:QG-EB-PL}, and the next point $\xknext$ also satisfies $\xknext \in  [f\leq f^\star + \nu]$ as $f(\xknext) \leq  f(\xknext) + \frac{1}{2c_k}\|\xknext - \xk\|^2 \leq f(\xk) $ by the update rule \cref{eq:PPM}. For the analysis below, we assume $\xk \in  [f\leq f^\star + \nu]$. 
    
    We next show that \cref{eq:PL} gives a simple proof of \cref{eq:cost-gap-linear}, and \cref{eq:QG} together with \cref{eq:EB} leads to a clean proof of \cref{eq:iterate-linear}. Recall that the optimality condition of \cref{eq:PPM} directly implies 
\begin{equation}
    - (\xknext - \xk)/c_k  \in \partial f(\xknext).\label{eq:optimality-condition}
\end{equation}  
    Then, the following inequalities hold 
    \begin{equation}
        \begin{aligned} \label{eq:PL-cost-gap}
         f(\xk) - f(\xknext) &\overset{(a)}{\geq} \|\xknext - \xk\|^2 /(2c_k)  \\
        &\overset{(b)}{\geq} \frac{c_k}{2} \Dist^2(0,\partial f(\xknext)) 
         \overset{(c)}{\geq} c_k \mup   (f(\xknext) - f^\star),
    \end{aligned}
    \end{equation}
    where $(a)$ applies the fact that $\xk$ is a suboptimal solution to \cref{eq:PPM}, $(b)$ comes from the optimality \cref{eq:optimality-condition}, and $(c)$ applies \cref{eq:PL}.
    Re-arranging and subtracting $f^*$ from both sides of \cref{eq:PL-cost-gap} lead to the desired linear convergence in \cref{eq:cost-gap-linear}. 
    
   We next use \cref{eq:QG} to prove \cref{eq:iterate-linear} with coefficient $\theta_k \leq 1/\sqrt{c_k\mu_{\mathrm{q}} +1}$. 
   By definition, we have $f(\Pi_{S}(\xk)) = f^\star$ and $\|  \Pi_{S}(\xk) - \xk\|^2 = \Dist^2(\xk,S)$. Since $f + \frac{1}{2c_k} \|\cdot- \xk\|^2$ is  $1/c_k$ strongly convex, its first-order lower bound at $\xknext$ is 
   \begin{equation} 
   \label{eq:first-order-lower-bound}
   \begin{aligned}
       & f^\star +\frac{1}{2c_k} \|  \Pi_{S}(\xk) - \xk\|^2  = f( \Pi_{S}(\xk) )  +\frac{1}{2c_k} \Dist^2(\xk,S) \\
       \geq ~&  f(\xknext) + \frac{1}{2c_k} \|\xknext- \xk\|^2 + \frac{1}{2c_k} \|\Pi_{S}(\xk) -\xknext \|^2, 
   \end{aligned}
   \end{equation}
where we also applied the fact that  $\xknext$ minimizes \cref{eq:PPM} so 0 is a subgradient. From \cref{eq:first-order-lower-bound}, we drop the positive term $\|\xknext- \xk\|^2$ and use the fact that $\|\Pi_{S}(\xk) -\xknext \| \geq \Dist(\xknext,S)$, leading to 
\begin{equation*}
        f^\star -  f(\xknext) + \Dist^2(\xk,S)/(2c_k) \geq  \Dist^2(\xknext,S)/(2c_k). 
\end{equation*} 
Combining this inequality with \cref{eq:QG} and simple re-arranging terms leads to the desired linear rate 
\begin{equation*}
\Dist^2(\xknext,S) \leq 1/\sqrt{c_k\muq + 1} \cdot \Dist^2(\xk,S).
\end{equation*}
Simple arguments based on \cref{eq:EB} can establish \cref{eq:iterate-linear} with coefficient $\theta_k \leq 1/(\mue^2/c_k^2 + 1)^{1/2}$. We provide some details in \cref{proof-thm-linear}.
\end{proof}

Two nice features of \Cref{thm-linear} are 1) the simplicity of its proofs and 2) the generality of its conditions. Indeed, the proof of \cref{eq:cost-gap-linear} is simple via \cref{eq:PL}, and the proof of \cref{eq:iterate-linear} is clean via  \cref{eq:QG} and \cref{eq:EB}, which are simpler than typical proofs. In addition, the regularity conditions are weaker than \cite{rockafellar1976monotone} and \cite{luque1984asymptotic}. Linear convergence for $\Dist(\xk,S)$ was first established in \cite{rockafellar1976monotone} with one restrictive assumption: the inverse of the subdifferential $(\partial f)^{-1}$ is locally Lipschitz at $0$, which requires a unique optimal solution, i.e.,~$S$~is a singleton.  The uniqueness assumption is lifted in \cite{luque1984asymptotic}, which even allows an unbounded solution set. More recently, this assumption is further relaxed to $\partial f$ being metrically subregular in \cite{cui2016asymptotic} and \cite{leventhal2009metric}.  
It is known that for convex functions, $\partial f$ is metrically subregular if and only if $f$ satisfies quadratic growth (cf. \cref{prop:QG-EB-PL}). Indeed, our proof in \cref{thm-linear} is based on quadratic growth, which is a more intuitive geometrical
property. Our main idea in the proof above is partially inspired by a recent result for the linear convergence of the spectral bundle method in \cite{ding2023revisiting,liao2023overview}.

We note that in the proof of \cref{thm-linear}, once the iterate is within the \cref{eq:PL}/\cref{eq:QG} region, the proof idea follows simply on the existence of the optimality condition \cref{eq:optimality-condition}. Building upon this observation, we extend \cref{thm-linear} to the class of weakly convex functions when a proper initialization is given.
\begin{theorem}
[Extension to weakly convex functions]
    \label{thm:linear-weakly} 
     Let $f\!:\! \RR^n \!\to \! \overline{\RR}$ be a proper closed $\rho$-weakly convex function with 
     $S = \argmin_{x \in \RR^n}f(x)$ and $S\neq \emptyset$. Suppose $f$ satisfies 
    \cref{eq:QG} with $\muq > \rho$ over the sublevel set $[f \leq f^\star + \nu]$ with $\nu > 0$. Let $\beta = \muq - \rho > 0$. The iterates generated by the PPM \cref{eq:PPM} with any positive sequence $\{c_k\}_{k \geq 0}$ satisfying $1/c_k > \rho $ for all  $k \geq 0$ and an initialization $x_0 \in [f\leq f^\star + \nu]$ satisfy 
    \begin{subequations}
    \begin{align*}
        f(\xknext) - f^\star & \leq  \omega_k \cdot (f(\xk) - f^\star), \\ 
        \Dist(\xknext,S) & \leq \theta_k \cdot \Dist(\xk,S), 
    \end{align*}
    \end{subequations}
    where the constants are
    \begin{equation*}
        \begin{aligned}
            \omega_k = \frac{1}{1+\mu_{\mathrm{p}} c_k }  < 1~\text{and}~\theta_k = \min \left \{ \frac{ 1}{\sqrt{ c_k\beta +1} } , \frac{1}{\sqrt{c_k^2/\mue^2 + 1}} \right \}< 1,
        \end{aligned}
    \end{equation*}
with $\mu_{\mathrm{p}}$ and $\mue$ being any valid constants for \Cref{eq:PL} and \Cref{eq:EB} from \Cref{table:equivalent-condtions}, respectively.
\end{theorem}
\begin{proof}
We first note that choosing $\{c_k\}_{k \geq 0}$ such that $1/c_k > \rho, \forall k \geq 0 $, guarantees the existence of the  optimality condition  $- (\xknext - \xk)/c_k  \in \fpartial f(\xknext)$
similar to \cref{eq:optimality-condition}. With the existence of the optimality condition, the proof follows similar arguments in 
the proof of \cref{thm-linear}. Specifically, the proof of the linear decrease in the cost value gap is identical to that in \cref{thm-linear}. As $\muq > \rho$, \cref{eq:EB}, \cref{eq:RSI}, and \cref{eq:PL} hold by \cref{prop:QG-EB-PL}. For the proof of the linear decrease in the distance, the coefficient $(c_k^2/\mue^2 + 1)^{-1/2}$ follows the same argument as in the convex case since $c_k$ is chosen such that $1/c_k > \rho $ so the function $f + \frac{1}{2c_k}\| \cdot - \xk \|^2 $ is a convex function plus a quadratic term and the proximal mapping satisfied the key inequality in \cref{lemma:firmly-monotone-conseuqence} in \cref{proof-thm-linear}; For the coefficient $(c_k\beta +1)^{-1/2} $, the argument is essentially the same as the convex case with a small modification that the function $f + \frac{1}{2c_k}\| \cdot - \xk \|^2 $ is $(\frac{1}{c_k} - \rho)$-strongly convex. 
\end{proof}

Note that \cref{thm:linear-weakly} requires the initial point to be within the sublevel set that satisfies \cref{eq:QG} and $\muq > \rho$. If the function $f$ satisfies \cref{eq:QG} and $\muq > \rho$ globally (i.e., $\nu = + \infty$), the above linear convergence holds with any initial condition. Let us illustrate the linear convergence for weakly convex function in the following example.  

\begin{example}[PPM for a weakly convex function] \label{example:weakly-convex-function}
    Consider the 2-weakly convex function $f$ in \cref{example:weakly-convex}. It satisfies \cref{eq:QG} and $6 = \muq > \rho = 2$ globally. Therefore, any initial point $x_0 \in \RR$ with a positive sequence $\{c_k\}_{k\geq 0}$ such that $1/c_k > \rho$ for all $k \geq 0$ guarantees the linear convergence. 
    Let $c_k = 1/6 $ for all $k \geq 0$. For any $x_k \in \RR$, from a simple calculation, the PPM update \cref{eq:PPM} follows the formula 
    \begin{equation*}
        x_{k+1} = \argmin_{ v \in  \{v_1,v_2\}} f(v) + 3 (v-x_k)^2,
    \end{equation*}
    where 
    \begin{equation*}
        v_1 = \begin{cases}
            -0.5, &~\text{if}~\frac{3}{2} x_k \geq  -0.5 \\
            \frac{3}{2} x_k,&~\text{if}~ -1 < \frac{3}{2}x_k < -0.5 \\
            -1, &~\text{if}~\frac{3}{2} x_k \leq -1
        \end{cases}
    \end{equation*}
    and 
    \begin{equation*}
        v_2 = \begin{cases}
            \argmin_{x \in \{-1,-0.5\} } f(x) + 3(x-x_k)^2, &~\text{if}~-1 < \frac{x_k-1}{2} < -0.5 \\
            \frac{x_k-1}{2},&~\text{otherwise}.
        \end{cases}
    \end{equation*} 
    We present the numerical result with the initial point $x_0 = -2$ in \cref{figure:weakly-convex-linear}. It shows the linear convergence in both the distance to the solution set and the cost value gap.
\end{example}
\begin{figure}[t]
    \centering
    \setlength{\abovecaptionskip}{0pt}
    \begin{subfigure}[b]{0.3\textwidth}
         \centering
         \includegraphics[width=\textwidth]{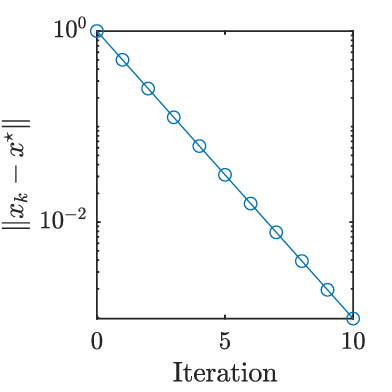}
     \end{subfigure}
     \hspace{20 mm}
     \begin{subfigure}[b]{0.3\textwidth}
         \centering
         \includegraphics[width=\textwidth]{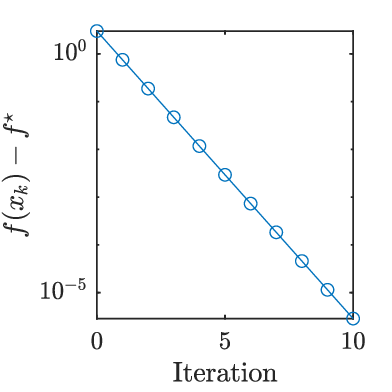}
     \end{subfigure}
    \caption{Linear convergences of the distance to the solution set (left) and cost value gaps (right) for a $2$-weakly convex function in \Cref{example:weakly-convex-function}.}
    \label{figure:weakly-convex-linear}
\end{figure}

\vspace{-1 mm}

\section{Inexact proximal point method (iPPM) and its convergence}
\label{section:inexact-PPM}
In \cref{section:sublinear-linear-PPM}, each subproblem \cref{eq:PPM-subproblem} is solved exactly. This may not be practical since one still needs an iterative solver to solve \cref{eq:PPM-subproblem}, where stopping criteria naturally introduce errors.  
We here discuss an inexact version of PPM (iPPM, \cite{rockafellar1976monotone}) where the subproblem \cref{eq:PPM} is solved inexactly. The regularity conditions in \cref{prop:QG-EB-PL} also allow us to establish linear convergences of iPPM.  
\vspace{-1mm}
\subsection{iPPM and stopping criteria} 
\label{subsection:ippm-stopping-criteria}
We replace the exact update \cref{eq:PPM}  with an inexact update
\begin{equation}
    \begin{aligned}
            \label{eq:PPM-inexact}
        \xknext \approx \text{prox}_{c_kf}(\xk).
    \end{aligned}
\end{equation}
Two classical criteria suggested in Rockafellar's seminal work \cite{rockafellar1976monotone} 
are 
\begin{align}
    \|\xknext - \text{prox}_{c_kf}(\xk)\|& \leq \epsilon_k, \quad \textstyle \sum_{k=0}^\infty \epsilon_k < \infty,
    \tag{A} \label{eq:criterion A}\\
    \|\xknext - \text{prox}_{c_kf}(\xk)\|& \leq \delta_k \| \xknext - \xk \|, \quad \textstyle \sum_{k=0}^\infty \delta_k < \infty. 
    \tag{B}
    \label{eq:criterion B}
\end{align}

The inexact update \cref{eq:PPM-inexact} with \cref{eq:criterion A} or \cref{eq:criterion B} is called iPPM. The two criteria are not implementable as the value of $\text{prox}_{c_kf}(\xk)$ is unknown. As discussed in \cite[Proposition 3]{rockafellar1976monotone}, two implementable alternatives that imply \cref{eq:criterion A,eq:criterion B}, respectively, are 
\begin{align}
    \Dist(0,H_k(\xknext)) &\leq \epsilon_k/c_k, \quad \textstyle \sum_{k=0}^\infty \epsilon_k < \infty,
    \tag{$\text{A}^{\prime}$} \label{eq:criterion Ap}\\
       \Dist(0,H_k(\xknext))& \leq  (\delta_k/c_k) \|\xknext - \xk\|, \quad \textstyle \sum_{k=0}^\infty \delta_k < \infty,  
     \tag{$\text{B}^{\prime}$} \label{eq:criterion Bp}
\end{align}
where $H_k(x) = \partial f(x) + (x-\xk)/c_k$ is the subdifferential of $f + \|\cdot-\xk\|^2/(2c_k)$ at $x$ (since $f$ is convex by assumption).
Note that \cref{eq:criterion A,eq:criterion B} only require the inexact update $\xknext$ to stay close enough to  $\text{prox}_{c_kf}(\xk)$ with respect to the Euclidean distance, but they do not require the inexact update $\xknext$ to be within the domain of $f$, i.e., $f(\xknext)$ might be infinity.
However, the stopping criteria
\eqref{eq:criterion Ap} and 
\eqref{eq:criterion Bp} require that $x_{k+1}$ is in the domain of $f$.

\subsection{(Sub)linear convergence of iPPM}
The seminal work by Rockafellar \cite{rockafellar1976monotone} has established the asymptotic convergence of iterates for iPPM under a general setup. We state the results below, whose proof is technically involved. 
\vspace{-0.05 pt}
\begin{theorem}[Asymptotic convergence of iterates {\cite[Theorem 1]{rockafellar1976monotone}}]
\label{thm:asymptotic-convergence}
Consider a proper closed convex function $f: \RR^n \to  \overline{\RR}$. Let $\{\xk\}_{k\geq 0 }$ be any sequence by \cref{eq:PPM-inexact} under \cref{eq:criterion A} with a positive sequence $\{c_k\}_{k\geq 0 }$ bounded away from zero. Then, we have 1) the sequence $\{\xk\}_{k\geq 0 }$ is bounded if and only if there exists a solution to $0 \in \partial f(x)$, i.e., $S \neq \emptyset$; 2) if $S \neq \emptyset$, the whole sequence $\{\xk\}_{k\geq 0 }$ converges to an optimal point $x_{\infty} \in S$, i.e., $\lim_{k \to \infty} x_k = x_\infty$. 
\end{theorem}
We next establish a sublinear convergence of iPPM for cost value gaps. Our simple proof is based on the boundedness of the iterates from \Cref{thm:asymptotic-convergence} and a recent idea in \cite[Theorem 3]{lu2023unified}. We provide its proof details in \cref{section:proof-inexact-PPM}. 
\begin{theorem}[\textbf{Sublinear convergence of iPPM}]
\label{thm:sublinear-inexact}
        Let $f: \RR^n \!\to \! \overline{\RR}$ be a proper closed convex function with $S = \argmin_{x \in \RR^n}f(x)$ and $S \neq \emptyset$. 
        The iterates \cref{eq:PPM-inexact} under \cref{eq:criterion Ap} with a positive sequence $\{c_k\}_{k\geq 0}$ bounded away from zero
        converge to  $x^\star \in S$ asymptotically, and the cost value gaps 
        satisfy
        \begin{align*}
           \textstyle
           \min_{j = 0,\ldots,k}f(x_j) - f^\star \leq 
           (\Dist^2(x_0,S) + 2D\sum_{j=0}^{k-1} \epsilon_j)/(2\sum_{j=0}^{k-1}c_j),
        \end{align*}
        where $D$ is the diameter of iterates $\{\xk\}_{k \geq 0}$ which is bounded.
\end{theorem}
   {Note that \Cref{thm:sublinear-inexact,thm:asymptotic-convergence} can be viewed as the convergence counterpart for iPPM of \cref{theo:PPM-sublinear-convergence} with two major differences: 1) \Cref{thm:sublinear-inexact} deals with the best iterate,  unlike~the~last iterate in \cref{theo:PPM-sublinear-convergence} (the guarantee for the average $\bar{x}_k = \frac{1}{k} \sum_{j=1}^{k} x_j $ or weighted average 
   $\Tilde{x}_k = ({\sum_{j=0}^{k-1} c_j x_{j+1}})/({\sum_{j=0}^{k-1} c_j})$
   is also straightforward; see \cref{remark:average-iterate} in the appendix); 2) the convergence of cost values in \Cref{thm:sublinear-inexact} relies on the boundedness of iterates in \cref{thm:asymptotic-convergence} whose proof is technically involved \cite[Theorem 1]{rockafellar1976monotone}, while the convergence of iterates of exact PPM is established from the sublinear convergence of cost values in \cref{theo:PPM-sublinear-convergence}.} \label{subsection:linear-convergence-inexact}

 Similar to \cref{thm-linear}, linear convergence of iPPM can also be deduced when popular regularity conditions are assumed. However, the definitions of \cref{eq:SC}, \cref{eq:RSI}, \cref{eq:EB}, \cref{eq:PL}, and \cref{eq:QG} in \cref{section:equivalent-conditions} are stated over a sublevel set $[f \leq f^\star + \nu]$ with $\nu > 0$, which is convenient to estimate the iteration number to enter a certain region, such as the constant $k_0$ in \cref{thm-linear}; but those definitions become too stringent in the analysis. Indeed,  note that the proof in \cref{thm-linear} requires the next point $\xknext$ remains inside the \cref{eq:QG} or \cref{eq:PL} region. As the stopping criteria \cref{eq:criterion A,eq:criterion B} allow infeasible points, it is possible that $\xknext$ is infeasible (i.e., $f(\xknext) = +\infty$), making the same proof fail. However, the iterates are still approaching an optimal solution and the inequality $  \muq/2\cdot \Dist(\xknext,S) \leq f(\xknext) - f^\star$ with any $\muq > 0$ still holds automatically.
 
 To state the linear convergence result for iPPM, we here modify the definition of \cref{eq:QG} from a sublevel set to a neighborhood of an optimal solution. Precisely, we say $f$ satisfies QG at $x^\star \in S$, if there is a constant $\muq > 0$ and a neighborhood $\mathcal{U} \subseteq \RR^n$ containing $x^\star$ such that
\begin{equation}
\label{eq:QG-ball} \frac{\muq}{2} \cdot \Dist^2(x,S)\leq f(x) - f^\star, \quad \forall x \in \mathcal{U}.    
\end{equation}
Notice that we can also redefine RSI, EB, and PL using a neighborhood (similar to \cref{eq:QG-ball}) and show the equivalence between them. However, we note that the neighborhood $\mathcal{U}$ for different regularity conditions may be different. For clarity, we present our final technical result that shows the linear convergence of iPPM under \cref{eq:QG-ball}, which is the counterpart of \cref{thm-linear}. 

\begin{theorem}[\textbf{Linear convergence of iPPM}]
    \label{thm-linear-inexact}
     Let $f :  \RR^n  \to  \overline{\RR}$ be a proper closed convex function
     with 
     $S = \argmin_{x \in \RR^n}f(x)$ and $S\neq \emptyset$. Suppose $f$ satisfies \cref{eq:QG-ball} at all $x^\star \in S$. Let $\{\xk\}_{k \geq 0}$ be any sequence generated by iPPM \cref{eq:PPM-inexact} under \cref{eq:criterion A,eq:criterion B} with a positive sequence $\{c_k\}_{k \geq 0}$ bounded away from zero. Then,
    there exist a nonnegative $\theta_k < 1$ and a large $\bar{k} >0$ such that
    \begin{equation*}
        \label{eq:iPPM-linear-convergence}
        \Dist(\xknext,S)  \leq \hat{\theta}_k \Dist(\xk,S),\; \forall k \geq \bar{k},
    \quad \text{where}\quad
           \hat{\theta}_k = \frac{ \theta_k +2\delta_k}{1-\delta_k}<1.
    \end{equation*}
\end{theorem}

To prove \cref{thm-linear-inexact}, we first introduce an important inequality characterizing the one-step improvement, which was first established in \cite[Equation 2.7]{luque1984asymptotic}. We provide some details in \cref{Appendix-subsection-lemma-one-step} for completeness. 
\begin{lemma} [Inexact one-step improvement {\cite[Equation 2.7]{luque1984asymptotic}}]
    \label{lemma:inexact-PPM-property-1}
     Let $f\!:\! \RR^n \!\to \! \overline{\RR}$ be a proper closed convex function
     with 
     $S = \argmin_{x \in \RR^n}f(x)$ and $S\neq \emptyset$. Let $\{\xk\}_{k \geq 0}$ be any sequence by the iPPM \cref{eq:PPM-inexact} under \cref{eq:criterion B} with the positive sequence $\{c_k\}_{k\geq 0}$ bounded away from zero.~Then there exists a $\hat{k}\geq 0$ such that 
    \begin{align}
    \label{eq:inexact-PPM-distance-error-app}
        (1-\delta_k ) \Dist(\xknext,S) \leq 2\delta_k \Dist(\xk,S) + \Dist(\mathrm{prox}_{c_kf}(\xk),S), \; \forall k \geq  \hat{k},
    \end{align}    
    where $\delta_k < 1$. 
\end{lemma}

\noindent \textbf{Proof of \cref{thm-linear-inexact}:}
Thanks to \cref{eq:criterion A}, the convergence of $\{\xk\}_{k\geq 0}$ is guaranteed by \cref{thm:asymptotic-convergence}, i.e., the sequence $\{\xk\}_{k\geq 0}$ converges to an optimal solution $x^\star$.
 Let $\muq$ and $\mathcal{U}$ be the positive constant and the neighborhood that satisfy  
\begin{equation}
    \label{eq:QG-ball-appendix}
    \frac{\muq}{2}\cdot \Dist^2(x,S)\leq f(x) - f^\star, \quad \forall x \in \mathcal{U}  
\end{equation}
with $x^\star \in \mathcal{U}.$  Let $r$ be the largest radius such that $\mathbb{B}_{\mathbb{R}^n}(x^\star,r) \subseteq \mathcal{U}$. Since $\{\xk\}_{k\geq 0}$ converges to $ x^\star$ and stopping criterion \cref{eq:criterion A} is implemented, there exist two integers $k_1,k_2 \geq 0$ such that 
 \begin{equation*}
 \begin{aligned}
     \|\xk-x^\star\| <  \frac{r}{2},& \; \forall k \geq k_1,~\text{and}~
     \|x_{k+1} - \prox_{c_k f}(\xk)\| < \frac{r}{2},\; \forall k \geq k_2.
 \end{aligned}
 \end{equation*}
 Therefore, for all $k \geq \max\{k_1,k_2\}$, we have
 \begin{equation*}
     \|\prox_{c_k f}(\xk) - x^\star\| \leq \|\prox_{c_k f}(\xk) - \xknext\| + \|\xknext - x^\star \| < r,
 \end{equation*}
which implies the point $\prox_{c_k f}(\xk)$ satisfies \cref{eq:QG-ball-appendix}. The same analysis in \cref{thm-linear} shows
\begin{equation}
     \label{eq:inequality-QG-iPPM}
        \Dist(\text{prox}_{c_k f}(\xk),S) \leq  \theta_{k} \Dist(\xk,S),\; \forall k \geq \max\{k_1,k_2\},
 \end{equation}
 where $\theta_{k} = 1/\sqrt{c_k \mu_{\mathrm{q}} +1}  < 1$.  On the other hand, thanks to \cref{eq:criterion B,lemma:inexact-PPM-property-1}, there exists a $k_3 > 0$ such that 
    \begin{equation}
            \label{eq:inexact-PPM-distance-error}
    (1-\delta_k ) \Dist(\xknext,S) \leq 2\delta_k \Dist(\xk,S) + \Dist(\text{prox}_{c_k f}(\xk),S), \; \forall k \geq  k_3,
    \end{equation}
Consequently, combining \cref{eq:inequality-QG-iPPM} with \cref{eq:inexact-PPM-distance-error} yields $$\Dist(\xknext,S) \leq \frac{\theta_k+2\delta_k}{1-\delta_k} \Dist(\xk,S), \quad \forall k \geq \max\{k_1,k_2,k_3\}.$$
Finally, as  we know  $\delta_k \to 0$ and 
$ \theta_k < 1, $ for all $ k\geq \max\{k_1,k_2\},$ 
there exists a $\bar{k} \geq \max\{k_1,k_2, k_3\}$ such that 
$$\frac{\theta_k + 2 \delta_k}{1-\delta_k} < 1, \quad \forall k \geq \bar{k}.$$ This completes the proof. \qed

    In \cref{thm-linear-inexact}, criterion \cref{eq:criterion A} serves to guarantee that the iterates can reach the neighborhood $\mathcal{U}$ in \cref{eq:QG-ball} (cf.~the asymptotic convergence from \cref{thm:asymptotic-convergence}). If  \cref{eq:QG-ball} holds globally (i.e., $r =+\infty $), the same linear convergence result holds with \cref{eq:criterion B} only. Note that the convergence proof in \cref{thm-linear-inexact} is very modular. Indeed, thanks to the regularity conditions (i.e., \cref{eq:QG-ball}), our proof appears simpler and less conservative than typical proofs in the literature; see the discussions after \Cref{thm-linear}. 

\section{Applications}
\label{section:applications} 

In this section, we present numerical experiments demonstrating the linear convergence of PPM. Throughout the experiment, we use the modeling package YALMIP \cite{Lofberg2004} to formulate the subproblem \cref{eq:PPM} and call the conic solver MOSEK \cite{mosek} to solve it. Our code is available at 
\begin{center}
    \url{https://github.com/soc-ucsd/PPM_examples}
\end{center}

We consider three different applications of convex optimization in machine learning and signal processing: linear support vector machine (SVM) \cite{zhang2015stochastic}, lasso \cite{tibshirani1996regression}, and elastic-net \cite{zou2005regularization}, respectively. 
The problem formulations are
\begin{itemize}
    \item Linear SVM:
    \vspace{-2 mm}
\begin{equation*}
        \min_{x\in \RR^n}\; \frac{1}{n} \sum_{i = 1}^n \max \{0,1 - b_i(a_i^\tr x)\} + \frac{\rho}{2} \|x\|^2,
\end{equation*}
where $\rho > 0, a_i \in \RR^d, b_i \in \{-1,1\}, i = 1, \ldots, n.$
    \item Lasso ($\ell_1$-regularization):
     \vspace{-2 mm}
    \begin{equation*}
          \min_{x\in \RR^m}\quad \frac{1}{2} \| y - Ax \|^2 + \lambda \|x\|_{1},
    \end{equation*}
    where $A \in \RR^{n\times m}, y \in \RR^m ,$ and $ \lambda > 0$. 
    \item Elastic-Net ($\ell_1-\ell_2^2$-regularization):
     \vspace{-2 mm}
     \begin{equation*}
        \min_{x\in \RR^m}\quad  \frac{1}{2}\| y - Ax \|^2 + \lambda \|x\|_{1}+ \frac{\mu}{2}\|x\|^2,
    \end{equation*}
    where $A \in \RR^{n\times m}, y \in \RR^m ,$ and $ \lambda, \mu > 0$.
\end{itemize}

For linear SVM, we consider three datasets (``a1a'', ``Australian'', ``fourclass'') from \cite{chang2011libsvm}. The regularization term is set as $\rho = 1.$ The step size parameter of the PPM is chosen as $c_k = 1$ for all $k\geq 1$. For Lasso, we randomly generate $A \in \RR^{n \times m}$ and $ \hat{x} \in \RR^{m}$ with $s$ elements being zero ($s < m$), then construct the vector $y = A\hat{x}$. We consider problems with three different sizes: $(n=10,m=40,s=5),(n=20,m=50,s=10),$ and $(n=30,m=60,s=15) $. The regularization terms are set as $\rho = 10$. The step size parameter of the PPM is chosen as $c_k = 0.16$ for all $k\geq 1$. For Elastic-Net, we consider the same data in Lasso with the regularization terms $\rho =10$ and $\mu = 1.$ 
These problems satisfy the regularity conditions in \Cref{thm-linear}. For each application, we run the PPM on three different data sets. The numerical results are shown in \cref{figure:svm-cost-gap}, which confirms the fast linear convergence of the PPM. 

\vspace{-2mm}

\begin{figure}[t]
\setlength{\abovecaptionskip}{1pt plus 0.1pt minus 0.1pt}
    \centering
    \includegraphics[width=\textwidth]{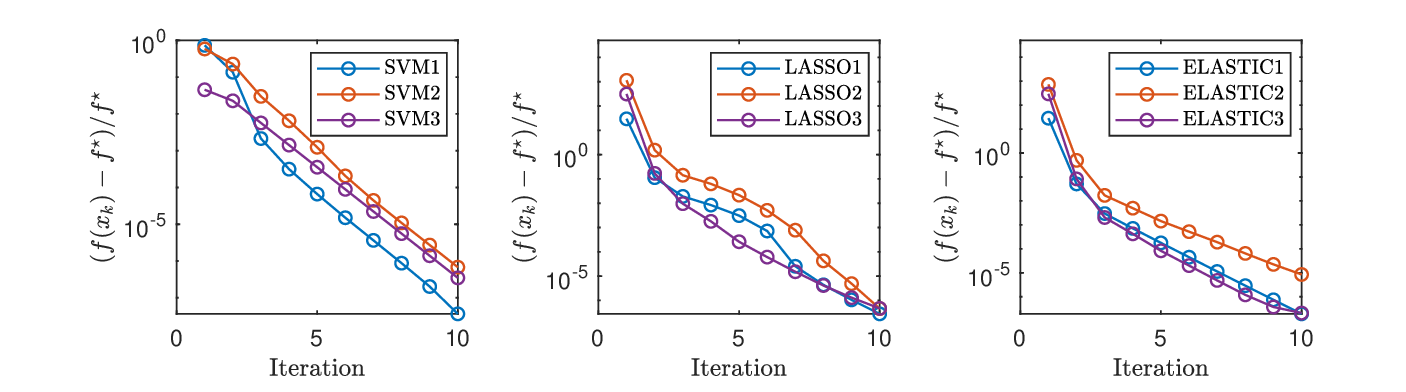}
    \caption{Linear convergences of cost value gaps for linear SVM (left), lasso (middle), and elastic-net (right).}
    \label{figure:svm-cost-gap}
\end{figure}

\section{Conclusion}
\label{section:conclusion}
In this paper, we have established the relationship between different regularity conditions under the class of $\rho$-weakly convex functions (see \cref{prop:QG-EB-PL}). The result shows that, in the class of $\rho$-weakly convex functions, {strong convexity} implies {restricted secant inequality}, which in turn implies {error bound}. Meanwhile, {error bound} and {Polyak-Łojasiewicz inequality} are equivalent, and any of them implies {quadratic growth}. If a $\rho$-weakly convex function satisfies {quadratic growth} and the quadratic growth constant satisfies $\muq > \rho$, then all these four properties: {restricted secant inequality}, {error bound}, {Polyak-Łojasiewicz inequality}, and {quadratic growth}, become equivalent. Unlike previous results in the literature, our result goes beyond smoothness and convexity, benefiting the analysis  of various iterative algorithms. 

We have also presented simple and modular proofs for the exact PPM for both convex and weakly convex functions which makes the analysis of the exact PPM more accessible (see \cref{thm-linear,thm:linear-weakly}). In the setting of inexact PPM, we first clarified the subtly in terms of inexactness and infeasibility, and then established the linear convergence of inexact PPM when the inexactness is controlled properly for convex optimization (see \cref{thm-linear-inexact}). 
The numerical experiments also validate our theoretical findings. We believe these results on (inexact) PPM will facilitate algorithm development in nonsmooth optimization, and we are particularly interested in further applications in large-scale conic optimization \cite{zheng2021chordal,zheng2020chordal}.

\appendix
\begin{appendices}
\section{Subdifferential characterization of weakly convex functions} \label{appendix:background}
For self-completeness, we review a useful subdifferential characterization of weakly convex functions. Recall that $\overline{\RR}$ denotes the extended real line, i.e., $\overline{\RR}:= \RR \cup \{\pm \infty\}$. 
    
    We say that a function $f: \RR^n \to \overline{\RR}$ is proper if $f(x) < \infty$ for at least one $x \in \RR^n$ and $-\infty < f(x)$ for all $x \in \RR^n.$     We say that a function $f: \RR^n \to \overline{\RR}$ is closed if the epigraph $\text{epi}f :\{ (x,\alpha) \in \RR^n \times \RR \mid f(x)\leq \alpha \}$ is closed.

\begin{definition}
    For a closed function $f: \RR^n \to \overline{\RR}$ and a point $x$ such that $f(x)$ is finite, we define the Fr\'echet subdifferential {at the point $x$ as}  
    $$
    \fpartial f(x) = \left \{ s \in \RR^n \mid \liminf_{y \to x }  \frac{ f(y) - f(x) -  \innerproduct{s}{y-x}}{\|y-x\|} \geq 0 \right\}.
    $$ 
\end{definition}

\begin{definition}
    A function $f:\RR^n \to \overline{\RR}$ is called $\rho$-weakly convex if $f + \frac{\rho}{2}\|\cdot\|^2$ is convex.
\end{definition}
For this class of $\rho$-weakly convex functions, the Fr\'echet subdifferential always exists for any point $x \in \RR^n$. This is because the Fr\'echet subdifferential is the same as the Clarke subdifferential for $\rho$-weakly convex functions and the Clarke subdifferential always exists \cite[Section 3 and fact 5]{li2020understanding}.
The class of $\rho$-weakly convex functions also has the following nice characterizations. These characterizations offer favorable properties for  $\rho$-weakly convex functions, which can be viewed as extensions from strongly convex functions. 
\begin{lemma}[Subdifferential characterization {\cite[Lemma 2.1]{davis2019stochastic}}]
    The following statements are equivalent for any closed and proper function $f:\RR^n \to \overline{\RR}$.
    \begin{itemize}
        \setlength\itemsep{0.1 pt}
        \item The function $f$ is $\rho$-weakly convex.
        \item The approximate secant inequality holds:
        \begin{equation*}
        \begin{aligned}
            f(\lambda x + (1-\lambda )y )\leq \lambda f(x) + (1-\lambda) f(y) +&\frac{\rho  \lambda (1-\lambda)}{2}\|x-y\|^2, \; \forall x,y \in \RR^n, \lambda \in [0,1].
            \end{aligned}
        \end{equation*}
        \item The subgradient inequality holds:
        \begin{equation} \label{eq:subdifferential-inequality-weakly-convex}
            f(y) \geq f(x) + \innerproduct{v}{y-x} - \frac{\rho}{2}\|y-x\|^2, \quad \forall x,y \in \RR^n, v \in \fpartial f(x).
        \end{equation}
        \item The subdifferential map is hypomonotone:
        \begin{equation*}
            \innerproduct{v-w}{x-y} \geq - \rho \|x-y\|^2, \quad \forall x, y \in \RR^n, v \in \fpartial f(x),~ \text{and}~w \in \fpartial f(y).
        \end{equation*}
    \end{itemize}
    If $f$ is $\mathcal{C}^2$-smooth, then the four properties above are all equivalent to 
    \begin{equation*}
        \nabla^2 f(x)  \succeq -\rho I, \forall x \in \RR^n.
    \end{equation*}
\end{lemma}
The subgradient inequality \cref{eq:subdifferential-inequality-weakly-convex} is particularly useful in our proof for \cref{prop:QG-EB-PL}. The proof repeatedly applies \cref{eq:subdifferential-inequality-weakly-convex} to establish the connection between subdifferential and cost value~gap.

\section{Proofs in \Cref{section:Motivation-Preliminaries}} \label{appendix:Section-2}

\subsection{Proof of \Cref{proposition:GD-linear-convergence}} \label{appendix:proof-GD}
Recall that we say a differentialable function $f:\RR^n \to \RR$ is $L$-smooth if it satisfies 
\begin{equation}
\label{eq:L-smooth-definion}
 \|\nabla f(x) - \nabla f(y)\| \leq L \|x  - y \|,\;\; \forall x, y \in \RR^n.   
\end{equation}
A differentialable $L$-smooth function $f$ also satisfy the following useful inequalities 
\begin{equation}
        \label{eq:Lispchitz-upper}
        f(y) \leq f(x) + \innerproduct{\nabla f(x)}{y-x} +\frac{L}{2}\|y-x\|^2, \;\; \forall x, y \in \RR^n. 
\end{equation}
Assume that RSI \cref{eq:RSI-smooth} holds with constant $\mur > 0$. We first note that $\mur \leq L$. Indeed, this can be seen from the following fact: let $x \in \RR^n, x \notin S, \hat{x} \in \Pi_{S}(x)$, we have
\begin{align*}
\mur \cdot \Dist^2(x,S)= \mur \|x-\hat{x}\|^2 & \leq \innerproduct{\nabla f(x)}{x-\hat{x}} =\innerproduct{\nabla f(x) - \nabla f(\hat{x})}{x-\hat{x}}  \\
& \leq L\|x-\hat{x}\|^2 = L \cdot \Dist^2(x,S),
\end{align*}
where the first inequality uses RSI \cref{eq:RSI-smooth}, and the last inequality applies Cauchy-Schwarz and \cref{eq:L-smooth-definion}.

Recall that for a closed set $S \subseteq \RR^n$, we denote the distance of a point $x\in \RR^n$ to $S$ as $\Dist(x,S) := \min_{y \in S} \|x - y\|$ and the projection of $x$ onto $S$ as $\Pi_{S}(x) = \argmin_{y\in S}\|x - y\|$. In the following, we use $S$ to denote the set of minimizers of $f$ and let $\hat{x}_k$ be an element in $ \Pi_{S}(\xk)$. 
We first have the following inequality in terms of the distance to the solution set 
\begin{subequations}
\begin{align}
    \Dist^2(\xknext,S) & \leq \|\xk - t_k \nabla f(\xk) - \hat{x}_k\|^2 \label{eq:iterate-decrease-a} \\ 
    & = \|\xk - \hat{x}_k\|^2- 2t_k \innerproduct{\nabla f(\xk)}{\xk - \hat{x}_k} + t_k^2\|\nabla f(\xk)\|^2 \label{eq:iterate-decrease-b} \\
    & \leq \Dist^2(\xk,S) - 2 t_k \mur \Dist^2(\xk,S) + t_k^2 L^2 \Dist^2(\xk,S)  \label{eq:iterate-decrease-c}\\
    & = \left(1- 2 t_k \mur + t_k^2 L^2\right)\Dist^2(\xk,S), \label{eq:iterate-decrease-d}
\end{align}
\end{subequations}
where \cref{eq:iterate-decrease-a} plugs in definition of $\xknext$ and uses the fact $\Dist(\xknext,S)  \leq \|\xknext - \hat{x}_k\| $, \cref{eq:iterate-decrease-b} simply expends the square, and  \cref{eq:iterate-decrease-c} uses \cref{eq:RSI-smooth} to upper bound the inner product term and applies the $L$-smoothness inequality \cref{eq:L-smooth-definion} to $\|\nabla f(\xk)\| = \|\nabla f(\xk) -  \nabla f(\hat{x}_k) \| $ $ \leq L \| \xk - \hat{x}_k \|$ (noting that $\nabla f(\hat{x}_k) = 0$). 

On the other hand, suppose the PL property \cref{eq:PL-smooth} holds with constant $\muq > 0$.
First, we have 
\begin{align*}
     f(\xknext) -f(\xk) &\leq  \innerproduct{\nabla f(\xk)}{\xknext-\xk} +\frac{L}{2}\|\xknext-\xk\|^2 \nonumber \\
     &=  -t_k\innerproduct{\nabla f(\xk)}{\nabla f(\xk)} + \frac{Lt_k^2}{2}\|\nabla f(\xk)\|^2 \nonumber  \\
     & = \left( \frac{-2t_k+Lt_k^2}{2}\right)\|\nabla f(\xk)\|^2.
\end{align*}
where the first inequality applies  \cref{eq:Lispchitz-upper}. If the step size $t_k$ is chosen such that $ 0< t_k< \frac{2}{L}$, then $ \frac{-2t_k+Lt_k^2}{2} < 0$.  Then, applying PL \cref{eq:PL-smooth} to the above inequality and subtracting $f^\star$ from both sides, we arrive at the linear convergence in the cost value gap 
\begin{equation}
    \label{eq:linear-decrease-cost}
    f(\xknext) - f^\star \leq (1+\left(-2t_k+Lt_k^2\right) \mup) (f(\xk) - f^\star).
\end{equation}

To summarize, if we choose the step size $t_k$ such that the constants in \cref{eq:iterate-decrease-d,eq:linear-decrease-cost} less than one (equivalently, $ 0 \leq 1- 2 t_k \mur + t_k^2 L^2 <1$ and $0< t_k< \frac{2}{L} $), linear convergence for both the distance and the cost value gap is guaranteed.
In particular, choosing $t_k = \frac{\mur}{L^2}$, we have $ 1- 2 t_k \mur + t_k^2 L^2  = 1 -  \frac{\mur^2}{L^2} < 1 $ and $ t_k \leq \frac{L}{L^2} < \frac{2}{L}$ as $\mur \leq L$, leading to the desired linear convergence:
\begin{equation*}
\begin{aligned}
    \Dist^2 (\xknext,S) & \leq (1 - \mur^2/L^2 )  \Dist^2 (\xk,S), 
    \\
    f(\xknext) - f^\star & \leq \frac{L^3 + (- 2 \mur L + \mur^2)  \mup }{L^3} (f(\xk) - f^\star).
    \end{aligned}
\end{equation*}
Note that if \cref{eq:PL-smooth} holds with $\mup >0$, then $ \mup \leq \frac{L^3}{2\mur L - \mur^2}$. Otherwise, the above cost value decrease would contradict the fact that the cost value gap is always nonnegative.
This choice of step size $t_k = \frac{\mur}{L^2}$ is not new, which was used  in \cite[Proposition 1]{guille2022gradient} or implicitly \cite[Proposition 1]{zhang2020new}. 

\begin{remark}
Note that \cref{proposition:GD-linear-convergence}
is consistent with the linear convergence result in \cite[Theorem 1]{karimi2016linear} since the step size $t_k = \frac{1}{L}$ also satisfies $0 < t_k < \frac{2}{L}$. Despite choosing $t_k = \frac{1}{L}$ renders a cleaner reduction constant $(1-\frac{\mur}{L})$ in \cite[Theorem 1]{karimi2016linear}, it is unclear if the distance also converges linearly under this stepsize choice (i.e., it is not guaranteed that $t_k = \frac{1}{L}$ leads to the constant $1- 2 t_k \mur + t_k^2 L^2 = 2 - 2 \frac{\mur}{L}<1$).
\end{remark}

\subsection{Proof of \Cref{lemma:slope-subdifferential}} \label{appendix:proof-lemma-2.1}
We note that the following inequality holds for any $\bar{x} \in \RR^n$ with a finite function value $f(\bar{x})$ \cite[Section 2]{drusvyatskiy2021nonsmooth}
\begin{equation}
    \label{eq:slope-Frechet-subdifferential}
    |\nabla f|(\bar{x}) \leq \Dist (0,\fpartial f(\bar{x}) ).
\end{equation}
As the slope lacks basic lower-semicontinuity properties, it is important to define the \textit{limiting slope}
\begin{equation}
    \begin{aligned}
        \overline{|\nabla f|}(\bar{x}) := \liminf_{x\to \bar{x},  f(x) \to f(\bar{x})} |\nabla f|(x). \label{eq:limit-slope}
    \end{aligned}
\end{equation}
The result from \cite[Proposition 8.5]{ioffe2017variational} has shown that
\begin{equation}
    \begin{aligned}
        \label{eq:limiting-slope-limiting-subdifferential}
        \overline{|\nabla f|}(\bar{x}) = \Dist(0,\partial f(\bar{x})),
    \end{aligned}
\end{equation}
where $\partial f(\bar{x})$ is the \textit{limiting subdifferential} \cite[Equation 18]{li2020understanding} at $\bar{x}$ defined as 
\begin{equation*}
    \partial f(\bar{x}) = \{s \in \RR^n \mid \exists x_k \to \Bar{x}~\text{and}~s_k \in \fpartial f(x_k)~\text{such that}~s_k \to s \}.
\end{equation*}

 Then we claim that for the class of functions whose Fr\'echet subdifferential and limiting subdifferential coincide, we have $|\nabla f|(\bar{x}) = \Dist(0,\fpartial f (\bar{x})).$
Indeed, suppose $\fpartial f(\bar{x}) = \partial f(\bar{x}) $, using \cref{eq:limiting-slope-limiting-subdifferential} and \cref{eq:slope-Frechet-subdifferential}, we have 
\begin{equation*}
    \overline{|\nabla f|}(\bar{x}) = \Dist(0,\partial f(\bar{x})) =  \Dist(0,\fpartial f(\bar{x})) \geq |\nabla f|(\bar{x}).
\end{equation*}
On the other hand, $\overline{|\nabla f|}(\bar{x}) \leq |\nabla f|(\bar{x})$ always holds true by defintion \cref{eq:limit-slope}. Thus, $ |\nabla f|(\bar{x}) = \Dist(0,\fpartial f (\bar{x}))$. 

{Below, we show that Fr\'echet subdifferential and limiting subdifferential match for 1) closed and smooth; 2) closed and convex; 3) closed and weakly convex functions. We also assume that the function is proper. We first show that the result holds for a $\rho$-weakly convex closed function $f:\RR \to \overline{\RR}$, and the result for a proper convex closed function follows by setting $\rho = 0$. By the definition of $\partial f(\bar{x})$, it is clear that $\fpartial f(\bar{x})\subseteq \partial f(\bar{x})$ (choosing $x_k = \bar{x}$ and $s_k = s \in \fpartial f(\bar{x})$). On the other hand, if $s \in  \partial f(\bar{x})$, then we can find $\{x_k\} \to \bar{x},  \{s_k\} \to s,$ and $s_k \in \fpartial f(x_k)$. Thus, for all $y \in \RR^n$, we have
\begin{align*}
\innerproduct{s}{y-\bar{x}} - \frac{\rho}{2}\|y-\bar{x}\|^2 
& = \lim_{k \rightarrow +\infty} \innerproduct{s_k}{y-x_k} - \frac{\rho}{2}\|y-x_k\|^2\\
& = \liminf_{k \rightarrow +\infty} \innerproduct{s_k}{y-x_k} - \frac{\rho}{2}\|y-x_k\|^2\\
&\overset{(a)}{\leq} 
\liminf_{k\rightarrow +\infty} f(y) - f(x_k)\\ &\overset{(b)}{\leq}
f(y) - f(\bar{x}),
\end{align*}
where $(a)$ applies the subgradient inequality \cref{eq:subdifferential-inequality-weakly-convex}, and $(b)$ uses the closeseness of $f$ (i.e., lower semicontinuity $\liminf_{y \to \bar{x}} f(y) \geq f(\bar{x})$). Thus, Fr\'echet subdifferential and limiting subdifferential match for $\rho$-weakly convex closed functions.

For a smooth function $f$, it is also clear that $\{\nabla f(\bar{x})\} \subseteq \partial f(\bar{x})$ since $\{\nabla f(\bar{x})\} \subseteq \fpartial f(\bar{x}) \subseteq \partial f(\bar{x})$. On the other hand, let $s \in \partial f(\bar{x})$. We have 
$    x_k \to \bar{x} $ and $ \{\nabla f(x_k)\} \to s.   $
Since $f$ is smooth, the gradient $\nabla f$ is continuous, implying $\{\nabla f(x_k)\} \to \nabla f(\bar{x})$. Thus, we have $s = \nabla f(x)$. This proves that $\{\nabla f(\bar{x})\} = \fpartial f(\bar{x}) = \partial f(\bar{x})$. Thus, the result follows.}

\section{Convergence proofs of the PPM in \cref{subsection:sublinear-exact}}
\label{section:sublinear-PPM-proof}

\subsection{A compelete proof for \cref{thm-linear}}
\label{proof-thm-linear}
To finish the proof for \cref{eq:iterate-linear} with coefficient $\theta_k \leq 1/{\sqrt{c_k^2/\mue^2 + 1}}$. We first review a crucial but standard result about the proximal mapping.
\begin{lemma}[{\cite[Equation 3.1]{leventhal2009metric}}]
\label{lemma:firmly-monotone-conseuqence}
For any $c_k > 0$ in \cref{eq:PPM}, the following holds true
\begin{align}
    \Dist^2(\xk,S)  \geq \|\xknext - \Pi_{S}(\xk) \|^2 + \|\xk - \xknext\|^2 \label{eq:firmly-nonexpensive}.
\end{align}
\end{lemma}
\begin{proof}
As shown in \cite{rockafellar2021characterizing}, the proximal point operator for a convex function $f:\RR^n \to \overline{\RR}$ with $\alpha > 0$ is \textit{firmly nonexpensive}, i.e., $\forall x,y \in \RR^n,$
\begin{align}
    \|\text{prox}_{\alpha f}(x)- \text{prox}_{\alpha f}(y)\|^2 \leq \innerproduct{x-y}{\text{prox}_{\alpha f}(x)- \text{prox}_{\alpha f}(y)}.
    \label{eq:firmly-nonexpensive-general}
\end{align}
For notational convenience, let $P(x):= \text{prox}_{\alpha f}(x).$ We have 
\begin{align*}
    \|x  - y\|^2 & = \|P(x) + ( x -P(x) ) -( P(y) + ( y -P(y) ) )\|^2 \\
    & = \|P(x) -P(y)+ ( x -P(x) ) - ( y -P(y) ) \|^2 \\
    & = \|P(x) -P(y)\|^2 + \|( x -P(x) ) - ( y -P(y) )\|^2 \\
    &\quad  + 2\innerproduct{P(x) -P(y)}{( x -P(x) ) - ( y -P(y) )} \\
    & \geq \|P(x) -P(y)\|^2 + \|( x -P(x) ) - ( y -P(y) )\|^2,
\end{align*}
where the last inequality comes from the fact that the inner product is nonnegative by \cref{eq:firmly-nonexpensive-general}.

Setting $x = \xk $ and $y = \Pi_{S}(\xk)$, we have $\xknext = P(\xk)$, and $y = P(y)$, which leads to the desired result in  \cref{eq:firmly-nonexpensive}.
\end{proof}
We are ready to finish the proof of \cref{thm-linear}. With \cref{lemma:firmly-monotone-conseuqence}, it follows that
\begin{align*}
    \Dist^2(\xknext,S) & \leq \|\xknext - \Pi_{S}(\xk)\|^2 \\
     & \leq \|\xk - \Pi_{S}(\xk)\|^2  - \|\xk - \xknext\|^2 \\
     & \leq \Dist^2(\xk,S) - c_k^2\Dist^2(0,\partial f(\xknext)) \\
     & \leq \Dist^2(\xk,S) - \frac{c_k^2}{\mue^2} \Dist^2(\xknext,S),
\end{align*}
where the second inequality is from \cref{eq:firmly-nonexpensive}, the third inequality comes from the optimality condition \cref{eq:optimality-condition}, and the last inequality uses \cref{eq:EB}.
Rearranging terms, we have 
\begin{align*}
    \Dist^2(\xknext,S) \leq \frac{\mue^2}{c_k^2 + \mue^2} \Dist^2(\xk,S).
\end{align*}
Taking a square root yields the desired result. 

\section{Convergence proofs for the inexact PPM in \cref{section:inexact-PPM}}
\label{section:proof-inexact-PPM}

\subsection{Proofs of \cref{thm:sublinear-inexact}}
   The optimality condition of \cref{eq:PPM-inexact} and criteria \cref{eq:criterion Ap} implies that there exist $v_k \in \partial f(\xknext)$ and  $\theta_k \in \RR^n$ such that
    \begin{align*}
        0 =  v_k + (\xknext - \xk +  \theta_k)/c_k,\quad \|\theta_k\| \leq \epsilon_k.  
    \end{align*}
      By the definition of subdifferential for convex functions, we know
    \begin{align}
        f(x^\star) 
        & \geq f(\xknext) + \left\langle -\frac{1}{c_k} (\xknext - \xk + \theta_k),  x^\star - \xknext \right\rangle \nonumber  \\
        & \geq  f(\xknext) + \frac{1}{c_k}\left\langle  \xk - \xknext ,  x^\star - \xknext \right\rangle  - 
        \frac{1}{c_k}
        \|\epsilon_k\| \|x^\star - \xknext\|.\label{eq:subgradient-inequality}
    \end{align}
    On the other hand, we have 
    \begin{align}
         \|\xknext - x^\star \|^2
         & = \|\xk  - x^\star \|^2 - \|\xk-\xknext\|^2 + 2 \innerproduct{\xknext - \xk}{\xknext - x^\star} \nonumber\\
        & \leq \|\xk - x^\star \|^2 + 2 \innerproduct{\xknext - \xk}{\xknext - x^\star} ,\label{eq:progress-iterate-1}
    \end{align}
where the first identity follows from simple algebraic manipulations $\|\xk - x^\star \|^2 = \|\xk - \xknext + \xknext - x^\star \|^2 = \|\xk - \xknext\|^2 - 2 \langle \xknext - \xk, \xknext - x^\star \rangle + \|\xknext - x^\star \|^2$, and the second inequality drops a nonnegative term. Combining \eqref{eq:subgradient-inequality} with \eqref{eq:progress-iterate-1}  leads to 
    \begin{align*}
        \|\xknext - x^\star \|^2 \leq \|\xk - x^\star \|^2  - 2 c_k ( f(\xknext)-  f(x^\star)) + 2\epsilon_k \|x^\star - \xknext\|.
    \end{align*}
       This improvement holds for all $k \geq 0$. Summing above inequality from $0$ to $k-1$ yields
    \begin{align*}
        &\|\xk - x^\star \|^2 \leq \|x_0 - x^\star\|^2 - 2 \sum_{j = 0}^{k-1} c_j (f(x_{j+1}) - f(x^\star)) +  \sum_{j = 0}^{k-1} 2\epsilon_j \|x^\star -  x_{j+1}\| 
        \end{align*}
    This immediately leads to
        \begin{align*}
          2\sum_{j=0}^{k-1} c_j  \min_{i = 0,\ldots,k}f(x_i) 
          - f(x^\star) & \leq 2\sum_{j = 0}^{k-1}c_j (f(x_{j+1}) - f(x^\star))\leq \|x_0-x^\star\|^2 + 2D\sum_{j=0}^{k-1} \epsilon_j.
    \end{align*}
    Dividing both sides by $2 \sum_{j=0}^{k-1}c_j$ and letting $x^\star = \Pi_{S}(x_0)$ gets the desired inequality. The convergence of the iterate $\{\xk\}$ and the boundedness of the diameter $D$ is guaranteed by \cref{thm:asymptotic-convergence}.

\begin{remark}[Average/weighted iterate]
\label{remark:average-iterate}
    \Cref{thm:sublinear-inexact} provides a convergence guarantee for the best iterate as opposed to that for the last iterate in \cref{theo:PPM-sublinear-convergence}. The guarantee for the average $\bar{x}_k = \frac{1}{k} \sum_{j=1}^{k} x_j $ or the weighted average $\Tilde{x}_k = \frac{\sum_{j=0}^{k-1} c_j x_{j+1}}{\sum_{j=0}^{k-1} c_j} $ can also be shown using the inequality
    \begin{equation*}
      f(\bar{x}_k) - f^\star ~\text{and}~
    f(\Tilde{x}_k) - f^\star 
    \leq  \frac{ \sum_{j=0}^{k-1} c_j (f(x_{j+1}) - f^\star)}{\sum_{j=0}^{k-1} c_j}.  
    \end{equation*}
 The proof in \cref{thm:sublinear-inexact} borrows some proof idea  from \cite[Theorem 3]{lu2023unified} with the fact that \cref{eq:criterion Ap} implies the existence of $\theta_k \in \RR^n$ such that $\Dist(0,H_k(\xknext)) =  \frac{\|\theta_k\|}{c_k}   \leq  \epsilon_k$. \qed
\end{remark}

\subsection{Proof of \cref{lemma:inexact-PPM-property-1}}
\label{Appendix-subsection-lemma-one-step}
The proof is an adaptation of \cite[Section 2]{luque1984asymptotic}. For completeness, we reprove it here. Fix a $x\in \RR^n$. From the update \cref{eq:PPM}, we know  
\begin{equation*}
    \begin{aligned}
        & f(\mathrm{prox}_{c_k f}(x)) + \frac{1}{2c_k}\| \mathrm{prox}_{c_k f}(x)- x\|^2 \leq f(\Pi_{S}(x)) +  \frac{1}{2c_k}\| \Pi_{S}(x)- x\|^2 \\
       \Longrightarrow ~ &  \frac{1}{2c_k}\| \mathrm{prox}_{c_k f}(x)- x\|^2  \leq f^\star - f(\mathrm{prox}_{c_k f}(x)) + \frac{1}{2c_k}\| \Pi_{S}(x)- x\|^2 \\
       \Longrightarrow ~ & \| \mathrm{prox}_{c_k f}(x)- x\| \leq \Dist(x,S),
    \end{aligned}
\end{equation*}
where the last implication is due to $f^\star - f(\mathrm{prox}_{c_k f}(x)) \leq 0$. On the other hand, using triangle inequality, the nonexpensiveness of projection onto a convex set, and the above inequality, we have
    \begin{align}
       \forall x \in \RR^n, ~ \|x - \Pi_{S}(\text{prox}_{c_k f}(x)) \| &  \leq \|x - \Pi_{S}(x)\| + \| \Pi_{S}(x)-\Pi_{S}(\text{prox}_{c_k f}(x))\| \nonumber \\
        & \leq  \Dist(x,S) + \| x-\text{prox}_{c_k f}(x)\|  \nonumber \\
        & \leq 2 \Dist(x,S).
        \label{eq:difference}
    \end{align}
    Also, we have
    \begin{align*}
        &~\|\xknext -\Pi_{S}(\text{prox}_{c_k f}(x)) \| \\
        \leq &~\| \xknext  -\text{prox}_{c_k f}(\xk) \| + \| \text{prox}_{c_k f}(\xk)- \Pi_{S}(\text{prox}_{c_k f}(x))\| \\
        \leq &~\delta_k \|\xknext - \xk\| + \Dist(\text{prox}_{c_k f}(\xk),S) \\
        \leq &~\delta_k \| \xknext -\Pi_{S}(\text{prox}_{c_k f}(x))  \| +  \delta_k \| \xk -\Pi_{S}(\text{prox}_{c_k f}(x))  \| \\
        & + \Dist(\text{prox}_{c_k f}(\xk),S),
    \end{align*}
    where the second inequality is from the stopping criterion \cref{eq:criterion B}. As $\delta_k \to 0$, choosing $\hat{k}$ such that $\delta_k < 1$ for all $ k \geq \hat{k},$ we have 
    \begin{align*}
        & (1-\delta_k )\|\xknext -\Pi_{S}(\text{prox}_{c_kf}(x)) \| \\ 
        \leq~& \delta_k \| \xk -\Pi_{S}(\text{prox}_{c_k f}(x))  \| + \Dist(\text{prox}_{c_k f}(\xk),S).
    \end{align*}
    Finally, using the fact $\Dist(\xknext,S) \leq \|\xknext -\Pi_{S}(\text{prox}_{c_k f}(x)) \|$ and \cref{eq:difference} for $x = \xk$ completes the proof.

\end{appendices}

\section*{Declarations}

\bmhead{Conflict of interest}
The authors have no conflict of interest to disclose.
\bmhead{Ethics approval and consent to participate} 
This research does not involve any study of humans or animals.
\bmhead{Code availability}
Our open-source implementation is available at \url{https://github.com/soc-ucsd/PPM_examples}.

\bibliographystyle{unsrt}
\bibliography{reference}

\end{document}